\documentclass{article}
\usepackage{amsmath,amssymb,amsthm}
\usepackage{mathtools}
\usepackage{enumitem}
\usepackage{tikz}
\usepackage{parskip}
\usepackage{lipsum}
\usepackage{authblk}
\usepackage[margin=3cm]{geometry}
\usepackage[normalem]{ulem}
\usepackage[numbers]{natbib}
\usepackage{graphicx} 
\usepackage[bookmarks=false,colorlinks=true,citecolor=blue]{hyperref}
\usepackage[skipabove=8pt,skipbelow=4pt,innertopmargin=6pt,innerbottommargin=6pt]{mdframed}
\usepackage{titlesec}
\usepackage{orcidlink}
\usepackage{comment}
\usepackage{caption}
\usepackage{subcaption}
\usepackage[pagewise]{lineno}
\usetikzlibrary{decorations.pathreplacing,calligraphy}
\usetikzlibrary{positioning, calc,chains,fit,shapes}
\usetikzlibrary{quotes}
\usepackage[colorinlistoftodos]{todonotes} 

\titleformat{\subsection}
{\bfseries}{\thesubsection}{1em}{}
\renewcommand{\qed}{\qedsymbol}

\newcommand{\NPC}{NP-complete}
\newtheorem{theorem}{Theorem}[section]{\bfseries}{\itshape}
\newtheorem{lemma}[theorem]{Lemma}{\bfseries}{\itshape}
\newtheorem{conjecture}[theorem]{Conjecture}{\bfseries}{\itshape}
\newtheorem{corollary}[theorem]{Corollary}{\bfseries}{\itshape}
\newtheorem{observation}[theorem]{Observation}{\bfseries}{\itshape}
\newtheorem{proposition}[theorem]{Proposition}{\bfseries}{\itshape}
\newtheorem{claim}{Claim}[theorem]

\newtheorem{construction}{Construction}
\DeclarePairedDelimiter\ceil{\lceil}{\rceil}
\DeclarePairedDelimiter\floor{\lfloor}{\rfloor}

\tikzstyle{filled vertex}  = [{circle,draw=blue,fill=black!50,inner sep=1pt}]  
\tikzstyle{empty vertex}  = [{circle, draw, fill = white, inner sep=1.5pt, minimum width=1.5pt}]
\tikzset{
  corner/.style  = {fill=gray!30, draw= gray, thick, inner sep=2pt},
  b-vertex/.style = {fill=RubineRed, diamond, draw=blue, inner sep=1.5pt},
  a-vertex/.style = {draw=blue, fill=cyan, inner sep=2pt}
}

\newcommand{\BNP}{\textsc{BNP}}
\newcommand{\DP}{\textsc{D3P}}

\title{\textbf{Graph Burning: Bounds and Hardness}}
\author[1,2]{\normalsize \textbf{Dhanyamol Antony}\orcidlink{0000-0001-7875-3457}}
\author[1]{\textbf{L. Sunil Chandran}\orcidlink{0000-0001-5451-6975}}
\author[3]{\textbf{Anita Das}\orcidlink{0000-0002-0126-4716}}
\author[1]{\textbf{Shirish Gosavi}}
\author[1,4]{\textbf{Dalu Jacob}}
\author[1,5]{\textbf{Shashanka Kulamarva}\orcidlink{0009-0002-2982-6044}}
\affil[1]{\small Department of Computer Science and Automation, Indian Institute of Science, Bengaluru, India}
\affil[2]{School of Data Science, Indian Institute of Science Education and Research, Thiruvananthapuram, India}
\affil[3]{Department of Mathematics, Manipal Institute of Technology Bengaluru, Manipal Academy of Higher Education, Manipal, India}
\affil[4]{Department of Mathematics, Indian Institute of Technology Delhi, India}
\affil[5]{Graduate School of Informatics, Kyoto University, Kyoto, Japan}
\affil[ ]{Email: \texttt{dhanyamolantony@iisertvm.ac.in, sunil@iisc.ac.in, anita.das@manipal.edu, shirishgp@iisc.ac.in, dalujacob@maths.iitd.ac.in, kulamarva.shashanka.3k@kyoto-u.ac.jp}}
\date{}

\begin{document}

\maketitle

\begin{abstract}
    \noindent 
    \emph{Graph burning} is a discrete-time process that models the propagation of information in a network. Given an undirected graph whose vertices are initially unburned, the process evolves in discrete rounds. At each round, an unburned vertex is selected and burned, while any unburned vertex adjacent to a vertex burned in the previous round also becomes burned. The \emph{burning number} of a graph is the minimum number of steps to burn all its vertices. The \textsc{Burning Number problem} asks whether the burning number of an input graph $G$ is at most $k$.
    In this paper, we investigate the graph burning problem from both algorithmic and structural viewpoints. Although the problem is known to be NP-complete on interval graphs, we strengthen this result by proving that it remains NP-complete even when restricted to connected proper interval graphs.
    We also study the burning number of $P_k$-free graphs. Motivated by the well-known \emph{burning number conjecture}, which states that every connected graph of order $n$ has burning number at most $\lceil \sqrt{n}~\rceil$, we establish an improved upper bound for connected $P_k$-free graphs and show that this bound is tight up to an additive constant of $1$.
    Finally, we study two variants of the problem: \emph{edge burning} and \emph{total burning}. We establish fundamental relationships between these variants and the classical burning, and we determine the computational complexity of the corresponding decision problems.
    \medskip
    
    \noindent \textbf{Keywords}: \textit{Burning Number; Graph burning;  Proper interval graphs; $P_k$-free graphs; Edge burning; Total burning}
    \medskip
		
\end{abstract}


\section{Introduction}
The spread of information, influence, and contagion through networks is a fundamental problem in network science. Understanding how rapidly such processes propagate has applications ranging from social media and viral marketing to epidemiology and communication networks. Examples include the spread of news, rumors, epidemics, and viral content on online social networks. The driving principle here is that a node can immediately influence only its acquaintances or neighbors. Although the initial influence comes from a single source node, gradually, additional source nodes emerge at different locations within the network to spread the influence. Consequently, the network's structure together with the sequence of nodes appearing over time are the factors influencing the rate of spread, making it essential to examine how this rate varies. A mathematical model called \emph{graph burning} was introduced~\cite{Bonato2016Burning} to analyze the rate of information spread over a network, where the spread of information is modeled as a fire spreading through the network. The discrete-time process of \emph{graph burning} is defined as follows: We are given an undirected graph $G$ with all vertices unburned. Each vertex is either burned or unburned. In every round (time step), we choose one unburned vertex to burn, and simultaneously, any unburned vertex adjacent to a vertex burned in the previous round also becomes burned. A vertex once burned cannot be unburned. The process ends when all vertices become burned. The \emph{burning number} of a graph $G$, denoted $b(G)$, is the minimum number of rounds for the process to end. Note that until the process ends, it is always possible to choose such an unburned vertex. A sequence of vertices $B=(b_1,b_2,\ldots,b_k)$, where $b_i$ denotes the vertex chosen in round $i$, is called a \emph{burning sequence} for $G$ of length $k$, and the vertices $b_1,b_2,\ldots,b_k$ are called the corresponding \emph{burning sources} for $G$. Alternatively, the burning number of $G$ is the length of a shortest burning sequence for $G$.
Now, the decision version of the problem is formally defined as follows.

\begin{mdframed}
    \textbf{\MakeUppercase{Burning Number Problem} (\BNP)}\\
    \textbf{Input:} An undirected graph $G$ with $n$ vertices and a positive integer $k$.\\
    \textbf{Question:} Is $b(G) \le k$? 
\end{mdframed}

Although graph burning was recently introduced, Alon~\cite{Alon1992TransmittingCube} studied a similar message-transmitting problem in a communication-theoretic setting and found that the burning number of an $n$-dimensional hypercube is $\ceil*{\frac{n}{2}}+1$. This line of research was further extended~\cite{Ho1996Transmitting,Jwo1994TransmittingDelay,Liu1996Routing}.

Clearly $b(G) \le d(G)+1$ for any connected graph $G$, where $d(G)$ is the diameter of $G$. This bound does not need to be tight. For example, for a path $P_n$ on $n$ vertices, $b(P_n)=\ceil*{\sqrt{n}~}$~\cite{Bonato2016Burning}, whereas $d(P_n)=n-1$. It was shown in~\cite{Bonato2016Burning} that every connected graph $G$ of order $n$ satisfies $b(G) \le 2\ceil*{\sqrt{n}~}-1$, and the authors conjectured the following stronger upper bound:

\begin{conjecture}[\cite{Bonato2016Burning}]\label{conj:BurningNumber}
    If $G$ is a connected graph of order $n$, then $b(G) \le \ceil*{\sqrt{n}~}$.
\end{conjecture}

Conjecture~\ref{conj:BurningNumber}, commonly referred to as the \emph{burning number conjecture}~\cite{Bonato2016Burning}, has attracted considerable attention since its introduction. The general upper bound has been progressively improved over the years~\cite{Land2016BurningUpperBound, Bessy2018BoundsBurning, Mitsche2018Burning, DBLP:journals/jctb/NorinT24}, while stronger bounds have also been established for several special graph classes, including fence graphs~\cite{Bonato2021BurningFence}, spiders and path forests~\cite{Bonato2019BurningSpiderPathForest}, and caterpillars~\cite{Hiller2021BurningCaterpillar,Liu2020BurningCaterpillar}.

Since $b(G)\leq b(T)$ for every connected graph $G$ and any spanning tree $T$ of $G$~\cite{Bonato2016Burning}, it is sufficient to prove the conjecture for trees. Accordingly, the conjecture has been verified for several subclasses of trees, including spiders (trees with exactly one vertex of degree at least $3$)~\cite{Bonato2019BurningSpiderPathForest,Das2018BurningSpiders}, caterpillars~\cite{Hiller2021BurningCaterpillar}, and trees without degree $2$ vertices~\cite{DBLP:journals/gc/Murakami24}. Beyond trees, the burning number has also been investigated for several other graph classes, including Hamiltonian graphs~\cite{Bonato2016Burning}, grids~\cite{Gorain2023WBurning}, hypercubes~\cite{Mitsche2018Burning}, biconvex bipartite graphs~\cite{Antony2023SpanningCaterpillar}, and graph products~\cite{Mitsche2018Burning}.

The conjectured bound is attained by paths and cycles. Moreover, the proof of Theorem~\ref{thm:PathAttainsTheBound} yields an $O(n)$-time algorithm for constructing an optimal burning sequence of a path or cycle on $n$ vertices. Despite these advances, the conjecture remains open for general connected graphs.

\begin{theorem}[\cite{Bonato2016Burning}]\label{thm:PathAttainsTheBound}
    If $G = P_n$ or $G = C_n$, then $b(G) = \ceil*{\sqrt{n}~}$.
\end{theorem}

The computational complexity of the \BNP\ has also received considerable attention. Even for path forests (disjoint union of paths) and trees with maximum degree 3~\cite{Bessy2017BurningIsHard}, the problem is {\NPC}. Moreover, \BNP\ is {\NPC} even for a subclass of interval graphs, namely caterpillars with maximum degree~3~\cite{Liu2020BurningCaterpillar}. The problem is also extensively studied in the paradigms of parameterized algorithms~\cite{ashok2025burn,Janssen2020Burning,Kare2019ParamBurningAlgo,Kobayashi2022ParaComplexBurning} and approximation algorithms~\cite{Bessy2017BurningIsHard, Bonato2019ApprxAlgoBurning,Kamali2020BurningTwoWorlds,Mondal2022KBurningHardApprx}.

Graph burning is closely related to several propagation processes on graphs, including the \textit{firefighter problem}~\cite{Finbow2009Firefighter,Pralat2013SparseNotFlammable,Pralat2014GraphsBurnSlowly} and \textit{graph bootstrap percolation}~\cite{Adler2003Bootstrap,Balogh2012GraphPercolation,Bollobas1968Weakly}. Unlike graph burning, the firefighter problem seeks to prevent the spread of fire, whereas bootstrap percolation studies propagation under threshold rules. In addition, a variant of graph burning is also studied~\cite{Mitsche2017Burning,Roshanbin2016Burning}, where the burning sequence is chosen based on a probabilistic rule. For a general overview of recent findings on graph burning, see~\cite{Bonato2021BurningSurvey}.
 
Although two natural variants of graph burning, namely \emph{edge burning}~\cite{Mondal2022KBurningHardApprx} and \emph{total burning}~\cite{Moghbel2020TotalBurning}, have been studied in the literature, both have received comparatively less attention. In edge burning, we burn only the edges, and the fire spreads via neighboring edges (edges incident on a common vertex). On the other hand, in total burning, we burn both vertices and edges, and the fire spreads via both neighboring vertices and neighboring edges.

We now summarize our contributions to the graph burning problem and its variants.

\noindent\textbf{Our Results:}
The contributions of this paper are twofold. In particular, we address a few algorithmic and structural questions concerning graph burning and two of its variants. Our first contribution strengthens the existing hardness results for the \BNP.

\begin{itemize}
    \item For an \emph{interval graph} $G$ with diameter $d$, it is known that $\ceil{\sqrt{d+1}~} \le b(G) \le \ceil{\sqrt{d+1}~} +1$~\cite{Kare2019ParamBurningAlgo}. Since the \BNP\ is {\NPC} for connected interval graphs~\cite{Gorain2023WBurning}, it follows that even deciding whether $b(G)=\ceil{\sqrt{d+1}}$ or $b(G) = \ceil{\sqrt{d+1}~} +1$ is {\NPC} for an interval graph $G$. We strengthen this result by proving that the statement is true even for \emph{connected proper interval graphs}, which is a restricted subclass of interval graphs.
    
    \item Since the diameter of a $P_k$-free graph (graph without a path on $k$ vertices as an induced subgraph) $G$ with $n$ vertices is at most $k-2$, we have $b(G)\le k-1$. We substantially improve this bound by proving $b(G) \le \ceil*{\frac{k+1}{2}}$ for connected $P_k$-free graphs (Theorem~\ref{thm:PkFreeBound}) and thereby settling Conjecture~\ref{conj:BurningNumber} affirmatively for $P_k$-free graphs whenever $k < 2\ceil{\sqrt{n}~} - 1$.
\end{itemize}
    
Furthermore, we explore two natural variants of graph burning: edge burning and total burning. For a graph $G$, let $b_L(G)$ and $b_T(G)$ denote the minimum number of steps needed for the edge burning and total burning of $G$, respectively. Surprisingly, both variants can be interpreted as special cases of classical graph burning. To be specific, for any graph $G$, we have $b_L(G)=b(L(G))$ and $b_T(G)=b(T(G))$, where $L(G)$ and $T(G)$ are the \emph{line graph} and the \emph{total graph} of $G$, respectively. We further investigate the relationship between the classical burning number and these parameters $b_L(G)$ and $b_T(G)$, and derive several interesting consequences.

\begin{itemize}
    \item We prove that  $b(G) - 1 \le b_L(G) \le b(G) + 1$ (Theorem~\ref{thm:LineGraphBounds}). 
    Moreover, we improve the upper bound when $G$ is a tree $T$, i.e., we prove that $b_L(T) \le b(T)$ (Theorem~\ref{thm:LineGraphTreeBounds}). 
    
    \item We resolve a conjecture on the total burning problem~\cite{Moghbel2020TotalBurning}, on the relationship between $b_T(G)$ and $b(G)$ positively (Theorem~\ref{thm:TotalGraphBounds}), i.e., we prove that $b(G) \le b_T(G) \le b(G) + 1$.

    \item To the best of our knowledge, the algorithmic complexity of these two variants of burning is not known. In this paper, we prove that the \BNP\ is {\NPC} even for the \emph{line graphs of caterpillars} (Corollary~\ref{cor:Linecaterpillar}) and \emph{total graphs of bounded degree trees} (Theorem~\ref{thm:totalhard}).    
    These results not only give more insights into the complexity of the \BNP\ on the well-known graph classes, namely, line graphs and total graphs, but also imply that \emph{the edge burning problem} and \emph{the total burning problem} are {\NPC} for the respective graph classes. 
\end{itemize}
    
Both proper interval graphs and line graphs are two incomparable subclasses of \emph{claw-free} graphs (graphs without a claw or $K_{1,3}$ as an induced subgraph). The algorithmic complexity of the \BNP\ was unknown for claw-free graphs. Our results, namely, Theorem~\ref{thm:ProperNPC} and Corollary~\ref{cor:Linecaterpillar}, imply that the \BNP\ is {\NPC} even for these two incomparable subclasses of claw-free graphs.

\section{Preliminaries}\label{sec:Prelim}
All graphs considered in the paper are finite, simple, and undirected. For a graph $G$, the vertex set and the edge set are denoted by $V(G)$ and $E(G)$, respectively. Two vertices are called \emph{neighbors} if they have an edge between them. Similarly, two edges are called neighbors if they are incident on a common vertex. A vertex and an edge are neighbors if the edge is incident on the vertex. For $X \subseteq V(G)$ (resp. $X \subseteq E(G)$), the graph obtained from $G$ by removing the vertices (resp. edges) in $X$ is denoted by $G-X$. We refer to \cite{West2001IGT} for basic graph theoretic notations and definitions.

Let $u,v$ be two vertices in a graph $G$. The \textit{distance} between $u$ and $v$, denoted by $d_G(u,v)$, is the number of edges in a shortest path between $u$ and $v$ in $G$. When the graph $G$ is understood from the context, we denote the distance as $d(u,v)$. We call any path between $u$ and $v$ as \textit{$u$-$v$ path.} Let $k$ be a non-negative integer. We use $[k]$ to denote the set $\{1,2, \dots , k\}$. A path on $k$ vertices is denoted as $P_k$. The middle vertex of a path $P = P_{2k+1}$, denoted by $c(P)$, is the $(k+1)^{th}$ vertex from any endpoint of $P$. Note that $c(P)$ is the vertex that is equidistant from both endpoints in $P$. The \emph{diameter} of a graph $G$, $d(G) = max\{d_G(u,v): u,v \in V(G)\}$. For any vertex $v$ in a graph $G$, the \emph{$k^{th}$ closed neighborhood} of $v$ in $G$, denoted by $N_G^k[v]$, is the set $\{u\in V(G): d(u,v)\leq k\}$.

Let $B = (b_1, b_2, \dots , b_k)$ be a burning sequence of a given graph $G$. For any $i \in [k]$, the \emph{burning cluster} of $b_i$, denoted by $B_{c}(b_i)$, is the set of vertices in $N_G^{k-i}[b_i]$. Further, the set of vertices in $G$ that are burned in the $i^{th}$ step is denoted by $S_i(B)$. Similarly, the set of vertices in $G$ that are burned within the first $i$ steps (including the step $i$) is denoted by $S_i^-(B)$. Note that every vertex $v$ of $G$ must satisfy $d(v,b_i) \le k-i$ for some $i \in [k]$. Thus we have
\[V(G) = N_G^{k-1}[b_1] \cup N_G^{k-2}[b_2] \cup \dots \cup N_G^0[b_k].\]
Moreover, if $|B| = k = b(G)$, then $B$ is called an \emph{optimal burning sequence} for $G$. The following observation follows immediately from the proof of Theorem~\ref{thm:PathAttainsTheBound} provided in~\cite{Bonato2016Burning}.

\begin{observation}[\cite{Bonato2016Burning}]\label{obs:PathOptBurn}
    Let $n$ be a positive integer, and let $B=(b_1, b_2, \dots, b_n)$ be an optimal burning sequence for the path $P_{n^2}$. Then the burning clusters of $b_1, b_2, \dots, b_n$ are pairwise disjoint, and $B_{c}(b_i)$ contains exactly $2(n-i)+1$ vertices for every $i \in [n]$.
\end{observation}

Let $H$ be a subgraph of a graph $G$. Note that the burning number is not a monotone property with respect to subgraphs, i.e., it is not necessary that $b(H) \le b(G)$. For example, consider the cycle $C_{n-1}$ and the wheel $W_n$, which is obtained by adding a new vertex to $C_{n-1}$ and joining it to every vertex of the cycle. $b(C_{n-1}) = \ceil*{\sqrt{n-1}~}$, whereas $b(W_n) = 2$ for $n \ge 4$. Thus, although $C_{n-1}$ is a subgraph of $W_n$, we have $b(C_{n-1}) > b(W_n)$. If $d_H(x,y) = d_G(x,y)$ for every $x,y \in V(H)$, then $H$ is called an \emph{isometric subgraph} of $G$ \cite{Bonato2016Burning}, denoted by $H \le_{iso} G$. Moreover, the burning number is not monotonic even on isometric subgraphs. For example, $C_5 \le_{iso} W_6$, while $b(C_5) = 3 > 2 = b(W_6)$. The following theorem provides a condition on an isometric subgraph $H$ under which the burning number is monotonic.

\begin{theorem}[\cite{Bonato2016Burning}]\label{lem:BurnMonotonSub}
    Let $H \le_{iso} G$. If for every vertex $u \in V(G) \setminus V(H)$ and every positive integer $r$, there exists a vertex $u_r \in V(H)$ with $N_G^r[u] \cap V(H) \subseteq N_H^r [u_r]$, then $b(H) \le b(G)$.
\end{theorem}

We call an isometric subgraph $H$ of $G$ a \emph{good-monotonic subgraph} of $G$ if it satisfies the hypothesis of Theorem~\ref{lem:BurnMonotonSub}. Alternatively, Theorem~\ref{lem:BurnMonotonSub} says that if $H$ is a good-monotonic subgraph of $G$, then $b(H) \le b(G)$.

A \emph{caterpillar} is a tree in which the removal of all leaves results in a path, called a \emph{stem}. A \emph{star} is a graph of order $n$ with $n-1$ vertices of degree $1$ and one vertex of degree $n-1$.  A \emph{claw} is a star of order $4$. A \emph{spider} is a tree with exactly one vertex, called the branching vertex, of degree greater than two. A \emph{leg} of a spider is a path from the branching vertex to a leaf of the tree.  The \emph{line graph} of a graph $G$, denoted by $L(G)$, is the graph with the vertex set $E(G)$, where two vertices $x$ and $y$ are adjacent in $L(G)$ if and only if the corresponding edges $x$ and $y$ share a common endpoint in $G$. A vertex $v$ in a connected graph $G$ is said to be a \emph{cut vertex} if $G-v$ is not a connected graph. A \emph{block} is a maximal subgraph of a graph without any cut vertices.

\section{Proper Interval Graphs}
In this section, we establish the hardness result for connected proper interval graphs. 

In this section, we establish the hardness result for connected proper interval graphs. 

An \emph{interval representation} of a graph $G$ is a collection  $\{I_v\}_{v\in V(G)}$ of intervals on a real line such that for any pair of vertices $u,v\in V(G)$, we have $uv\in E(G)$ if and only if $I_u\cap I_v\neq \emptyset$. A graph is said to be an \emph{interval graph} if it has a corresponding interval representation. A \emph{proper interval graph} is an interval graph with an interval representation in which no interval is properly contained in another interval. It is well known that proper interval graphs are exactly the claw-free interval graphs \cite{RobertsIndiffGraphs1969}.

The \BNP\ is known to be \NPC\ for path forests~\cite{Bessy2017BurningIsHard}, a subclass of disconnected proper interval graphs. However, the hardness of the problem on a class of disconnected graphs does not necessarily imply the hardness on the corresponding connected graph class. For instance, the \BNP\ exhibits different algorithmic complexities on path forests and paths. Motivated by this distinction, we prove in Theorem~\ref{thm:ProperNPC} that the \BNP\ is \NPC\ even for connected proper interval graphs, by providing a reduction from the \textsc{Distinct 3-partition} (which is known to be {\NPC}~\cite{Garey1990TheoryNPComp}) problem.

\begin{mdframed}
    \textbf{\MakeUppercase{Distinct 3-partition} (D3P)}\\
    \textbf{Input:} A set of distinct natural numbers, $X = \{a_1, a_2, ..., a_{3n}\}$, such that $\sum_{i=1}^{3n} a_i = nB$ and $\frac{B}{4} < a_i < \frac{B}{2}$ for all $i \in [3n]$.\\
    \textbf{Question:} Does there exist a partition of $X$ into $n$ triples such that the sum of the numbers in each triple is $B$?
\end{mdframed}

The \DP{} is known to be strongly {\NPC} \cite{Garey1990TheoryNPComp, Hulett2008MultigraphRealizations}; that is, this problem is \NPC, even when restricted to the cases where $B$ is bounded above by a polynomial in $n$. We claim that the problem is also \NPC, even when the input is restricted to odd numbers.

\begin{lemma}
    The \DP{} is {\NPC} even when all the numbers in the input are odd.
\end{lemma}

\begin{proof}
    To prove this, we provide a reduction from the \DP{}, which is known to be {\NPC}. Let $X=\{a_1,a_2,\dots,a_{3n}\}$ be a \DP{} instance. Then the target sum is $B=\frac{1}{n}\cdot \sum_{i=1}^{3n}a_i$. Consider a new instance consisting only of odd natural numbers by defining $X'= \{a'_1,a'_2,\dots,a'_{3n}\} = \{2a_i+1:a_i \in X\}$. Note that the new target sum becomes $B' = 2B+3$, and since the numbers in $X$ are distinct, the numbers in $X'$ are also distinct. Further, for all $i \in [3n]$, since $\frac{B}{4} < a_i < \frac{B}{2}$, multiplying the inequality by $2$ and adding $1$ yields
    \[
    \frac{B}{2}+1 < 2a_i+1 < B+1.
    \]
    Observe that \[
    \frac{B'}{4}=\frac{2B+3}{4}
    =\frac{B}{2}+\frac{3}{4}
    < \frac{B}{2}+1
    \]
    and
    \[
    B+1 < B+\frac{3}{2}
    =\frac{2B+3}{2}
    =\frac{B'}{2}.
    \]
    Hence, we have
    \[
    \frac{B'}{4} < a'_i = 2a_i+1 < \frac{B'}{2}.
    \] 
    Therefore, $X'$ is also a \DP{} instance but restricted to odd numbers. Moreover, one can see that the transformation from $X$ to $X'$ can be done in polynomial time. Furthermore, for any triple $\{a_i,a_j,a_k\}$, we have the following.
    \begin{align*}
        & a'_i+a'_j+a'_k = (2a_i+1)+(2a_j+1)+(2a_k+1)=2(a_i+a_j+a_k)+3 \\
        \text{Thus, } & a_i+a_j+a_k = B \iff a'_i+a'_j+a'_k = 2B+3 = B'.
    \end{align*}
    Hence, $X$ is a YES-instance of \DP{} if and only if $X'$ is a YES-instance of \DP{}, as desired.
\end{proof}

The following proposition gives the bounds for the burning number of interval graphs.

\begin{proposition}[\cite{Kare2019ParamBurningAlgo}]\label{prop:interval}
    If $G$ is an interval graph with diameter $d$, then \[\ceil{\sqrt{d+1}~} \le b(G) \le \ceil{\sqrt{d+1}~} +1.\]
\end{proposition}

Since proper interval graphs are a subclass of interval graphs, the above bounds also hold for proper interval graphs. Proposition~\ref{prop:interval} follows from the fact that if $G$ is an interval graph with a diametral path $P$, then every vertex of $G$ either lies on $P$ or is adjacent to at least one vertex of $P$. Consequently, the burning number of $G$ is either $b(P)$ or $b(P)+1$. Hence, we establish in Theorem~\ref{thm:ProperNPC} that determining whether $b(G)=b(P)$ is {\NPC} even when $G$ is a connected proper interval graph.

\noindent\textbf{High-Level Description of the Reduction:} We reduce from the \DP{} restricted to odd integers. Given an instance $X$, we construct a proper interval graph $G_P$ using Construction~\ref{cons:proper}, which is obtained by modifying the ``comb'' structure used in the NP-completeness proof for interval graphs in \cite{Gorain2023WBurning}.

The restriction to odd integers is crucial because the reduction is based on the structural property that a path on $(2m+1)^2$ vertices can be burned optimally in $2m+1$ steps, and such a burning partitions the path into $2m+1$ disjoint burning clusters whose sizes are exactly the first $2m+1$ odd integers. This property allows us to encode the elements of $X$ as lengths of suitable subpaths of the path.

We first construct a path $P$ composed of several subpaths arranged so that their lengths correspond to the required odd integers. The subpaths corresponding to the elements of $X$ represent potential triples in a solution to the \DP{} instance, while the remaining subpaths account for the other odd integers. Intuitively, an optimal burning of $P$ corresponds to selecting triples of elements from $X$ whose sums equal the target sum in the \DP{} instance. We then modify certain subpaths of $P$ as in Construction~\ref{cons:proper}, producing a connected proper interval graph $G_P$. In the subsequent analysis, we show that $G_P$ can be burned in $2m+1$ steps if and only if $X$ is a YES-instance of \DP{}.

We now describe the construction that transforms an instance $X$ of \DP{} restricted to odd numbers into a connected proper interval graph $G_P$.

\begin{construction}
    \label{cons:proper}
    Let $X=\{a_1,a_2,\dots,a_{3n}\}$ be an arbitrary instance of the \DP{} restricted to odd numbers. Then $n=\frac{|X|}{3}$, $B=\frac{1}{n}\cdot \sum_{i=1}^{3n}a_i$, and each $a_i = 2x_i-1$ (i.e., $x_i^{th}$ odd number) for some corresponding natural number $x_i$. Let $m$ be such that the $m^{th}$ odd number (i.e., $2m-1$) is the maximum element in $X$, and let $k=m-3n$. Let $Z$ be the first $m$ odd natural numbers, i.e., $Z = \{1,3,\dots,2m-1\}$. Let $Y=Z\setminus X$. So, $|Y|=k$. We construct a proper interval graph $G_P$ from $X$ as follows:
    \begin{itemize}
        \item Introduce $n$ paths $S_1,S_2,\dots, S_n$ each with $B$ vertices, and $k$ paths $S_1',S_2',$ $\dots, S_k'$ such that each path $S_i'$ is of order $y_i$, where $y_i$ is the $i^{th}$ largest number in $Y$. Now, introduce another $m+1$ paths $Q_1,Q_2,\dots,Q_{m+1}$ such that each path $Q_j$ is of order $2(2m+1-j)+1$. Note that the orders of the $Q_j$'s are all odd numbers from $2m+1$ to $4m+1$. Further, construct a larger path $P$ by joining these paths $S_i$, $S_j'$, and $Q_l$, for $1\le i\le n$, $1\le j\le k$, and $1\le l\le m+1$, in the following order:
        \[S_1,Q_1,S_2,Q_2,\dots, S_n,Q_n,S_1',Q_{n+1},S_2',Q_{n+2},\dots,S_k', Q_{n+k}, Q_{n+k+1},\dots,Q_{m+1}.\]
        Note that the total number of vertices in $S_i$'s and $S_i'$'s is $nB = \sum_{i=1}^{3n}a_i$ and $\sum_{i=1}^ky_i$, respectively, which collectively is the sum of all the odd numbers in $[2m-1]$. Further, since the orders of the $Q_j$'s are all odd numbers from $2m+1$ to $4m+1$, we have: \[|P|=nB+\sum_{i=1}^ky_i+\sum_{j=1}^{m+1}(2(2m+1-j)+1)=(2m+1)^2.\]
       
       \item For each subpath $Q_i$, $1\le i\le m+1$ in $P$, introduce $i'-1$ vertices denoted by $q_{i1},q_{i2},\dots,q_{i(i'-1)}$, where $i' = |Q_i|$. Now, the adjacency between vertices in $Q_i$ and the new $i'-1$ vertices is defined as follows: for $1\le x\le |Q_i|$, the $x^{th}$ and the $(x+1)^{th}$ vertices of $Q_i$ are adjacent to the vertex $q_{ix}$. The subgraph obtained after this modification to the subpath $Q_i$ is denoted by $Q_i^P$ and is a subgraph of $G_P$, i.e., $Q_i^P = G_P[V(Q_i) \cup \{q_{i1},q_{i2},\dots,q_{i(i'-1)}\}]$.
   \end{itemize}
   This completes the construction. See Fig.~\ref{fig:proper} for an example.
\end{construction}

\begin{figure}[t]
    \input{Sections/figs/propereg}
    \caption{An example of Construction~\ref{cons:proper} for the \DP{} instance $X=~\{19,21,23,27,29,31\}$ with $m=16$. All horizontal lines collectively represent the induced path $P$ on $G_P$. For $1\le i\le m+1$, $1\le j\le i'-1$, and $i'=|Q_i|$, $q_{ij}$'s are the hanging vertices adjacent to the vertices on $P$. Note that $|S_1| = |S_2| = 75$, $|S'_1| = 1$, $|S'_2| = 3, \cdots |S'_9| = 17$, $|S'_{10}| = 25$, $|Q_1| = 33$, $|Q_2| = 35, \cdots |Q_{17}| = 65$.}
    \label{fig:proper}
\end{figure}

The following lemmas and observations based on Construction~\ref{cons:proper} are useful to prove Theorem~\ref{thm:ProperNPC}.

\begin{lemma}
    Transformation from an instance $X$ of the \DP{} restricted to odd numbers to the corresponding instance of the \BNP\ on a proper interval graph runs in polynomial time.
\end{lemma}

\begin{proof}
    Since \DP{} is strongly {\NPC}, the numerical values appearing in the instance $X$, specifically the odd integers $a_i$ and the target sum $B$, are bounded by a polynomial in the input size.
    The construction of the proper interval graph produces graph components whose sizes are linear in the values $a_i$ and $B$. Since these values are polynomially bounded, the total size of the constructed graph and the time required to build it are polynomial in the numerical parameters of the instance. Consequently, the overall transformation runs in polynomial time.
\end{proof}

\begin{observation}
    \label{obs:G_PBurningLowerBound}
    $b(G_P) \ge 2m+1$.
\end{observation}

\begin{proof}
    For every $x,y \in V(P)$, we have $d_P(x,y) = d_{G_P}(x,y)$, and for every vertex $u \in V(G_P) \setminus V(P)$ (precisely $q_{ij}$'s) and every positive integer $r$, there exists a vertex $u_r \in V(P)$ (a neighbor of $q_{ij}$) such that $N_{G_P}^r[u] \cap V(P) \subseteq N_P^r [u_r]$. Hence, $P$ is a good-monotonic subgraph of $G_P$. Thus, by Lemma~\ref{lem:BurnMonotonSub}, $b(P) \le b(G_P)$. Furthermore, $b(P) = 2m+1$ since $P = P_{(2m+1)^2}$, which implies the claim.
\end{proof}


\begin{observation}
\label{obs:optimalburn_proper_path}
    While burning the graph $G_P$, if a single burning source can burn a subpath $Q_i$ of $P$, for some $i \in [m+1]$, in $t$ steps, then the same source can also burn $Q_i^P$ in $t$ steps.
\end{observation}

\begin{proof}
    Suppose that while burning $G_P$, a single burning source $x \in V(G_P)$ burns $Q_i$ in $t$ steps for some $i \in [m+1]$. Then every vertex $u \in V(Q_i)$ satisfies ${d}(x,u) \le t-1$. Let $v$ be an arbitrary vertex in $V(Q_i^P) \setminus V(Q_i)$. If we show that ${d}(x,v) \le t-1$, then $x$ can also burn the vertex $v$ in $t$ steps, which in turn implies that $x$ can also burn all the vertices of $Q_i^P$ in $t$ steps. If $v = x$, then trivially, we have ${d}(x,x) = 0 \le t-1$, as desired. Hence, we have $v \neq x$. By Construction~\ref{cons:proper}, the vertex $v$ is adjacent to two vertices $u_1,u_2 \in V(Q_i)$, which are adjacent to each other. Thus, we have ${d}(x,u_1) \le t-1$, ${d}(x,u_2) \le t-1$, and ${d}(x,v) \le \min\{{d}(x,u_1),{d}(x,u_2)\} + 1$. Since $v \neq x$ and $v$ is the only vertex that is equidistant from $u_1$ and $u_2$, it is not possible to have ${d}(x,u_1) = {d}(x,u_2) = t-1$. Without loss of generality, we have ${d}(x,u_1) \le t-2$, implying that ${d}(x,v) \le t-1$, as desired.
\end{proof}

\begin{lemma}
    \label{lem:bcluster}
    If all the burning sources of a burning sequence for $G_P$ lie on the path $P$ with at least two burning sources on a subpath $Q_i$ for some $i \in [m+1]$, then the burning clusters corresponding to at least two of these burning sources have a nonempty intersection within $Q_i$.
\end{lemma}

\begin{proof}
    Suppose that all the burning sources of a burning sequence $B$ for $G_P$ lie on the path $P$, and for some $i \in [m+1]$, $Q_i$ contains at least two burning sources, say $u_1$ and $u_2$. In addition, $u_1$ and $u_2$ are chosen such that no other burning sources lie on the $u_1-u_2$ path in $Q_i$ (and hence in $G_P$). Suppose, for contradiction, that $B_c(u_1) \cap B_c(u_2) \cap V(Q_i) = \emptyset$. Let $v$ be the vertex of the path $Q_i$ in $B_c(u_1)$ that is closest to $u_2$, and let $v'$ be the vertex of $Q_i$ in $B_c(u_2)$ that is closest to $u_1$. Since $B$ is a burning sequence for $G_P$ and there are no other burning sources on the $u_1-u_2$ path in $Q_i$, we have $v \ne v'$ and $vv' \in E(Q_i)$. Let $v_s$ be the common neighbor of $v$ and $v'$ in $Q_i^P$ as shown in Fig.~\ref{fig:proper_sub}. Since $B_c(u_1) \cap B_c(u_2) \cap V(Q_i) = \emptyset$, we have $v \notin B_c(u_2)$ and $v' \notin B_c(u_1)$. Therefore, $v_s \notin B_c(u_1) \cup B_c(u_2)$, implying that the vertex $v_s$ is unburned in $G^P$, a contradiction since $B$ is a burning sequence for $G_P$. Consequently, the lemma holds.
\end{proof}
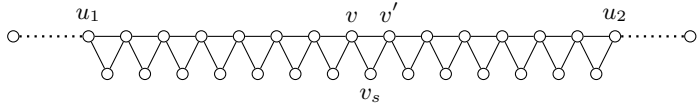
\begin{figure}[t]
    \tikzstyle{filled vertex}  = [{circle,draw=blue,fill=black!50,inner sep=1.5pt}]  
\tikzstyle{empty vertex}  = [{circle, draw, fill = white, inner sep=1.5pt}]
  \centering \small
\begin{tikzpicture}[scale=.5] 
{
      \def\n{4}
      \def\radius{3}      

      \node[empty vertex] (u) at  (-2,0){};
      \node[empty vertex,,label=above:\small{$u_1$}] (v1) at  (0,0){};
      \node[empty vertex] (v2) at  (1,0){};
      \node[empty vertex] (v3) at  (2,0){};
      \node[empty vertex] (v4) at  (3,0){};
      \node[empty vertex] (v5) at  (4,0){};
      \node[empty vertex] (v6) at  (5,0){};
      \node[empty vertex] (v7) at  (6,0){};
      \node[empty vertex,label=above:\small{$v$}] (v8)  at  (7,0){};
      \node[empty vertex,label=above:\small{$v'$}] (v9) at  (8,0){};
      \node[empty vertex] (v10) at  (9,0){};
      \node[empty vertex] (v11) at  (10,0){};
      \node[empty vertex] (v12) at  (11,0){};
      \node[empty vertex] (v13) at  (12,0){};
      \node[empty vertex] (v14) at  (13,0){};
      \node[empty vertex,,label=above:\small{$u_2$}] (v15) at  (14,0){};
      \node[empty vertex] (v) at  (16,0){};

      \node[empty vertex] (v1') at  (0.5,-1){};
      \node[empty vertex] (v2') at  (1.5,-1){};
      \node[empty vertex] (v3') at  (2.5,-1){};
      \node[empty vertex] (v4') at  (3.5,-1){};
      \node[empty vertex] (v5') at  (4.5,-1){};
      \node[empty vertex] (v6') at  (5.5,-1){};
      \node[empty vertex] (v7') at  (6.5,-1){};
      \node[empty vertex,label=below:\small{$v_s$}] (v8') at  (7.5,-1){};
      \node[empty vertex] (v9') at  (8.5,-1){};
      \node[empty vertex] (v10') at  (9.5,-1){};
      \node[empty vertex] (v11') at  (10.5,-1){};
      \node[empty vertex] (v12') at  (11.5,-1){};
      \node[empty vertex] (v13') at  (12.5,-1){};
      \node[empty vertex] (v14') at  (13.5,-1){};

      \foreach \i in {1,..., 14} {
        \pgfmathsetmacro{\p}{int(\i + 1)}
            \draw (v\i) -- (v\p);
        }

     \foreach \i in {1,..., 14} {
        \pgfmathsetmacro{\p}{int(\i + 1)}
            \draw (v\i) -- (v\i');
            \draw (v\i') -- (v\p);
        }
      
    \draw[dotted, line width=0.3mm] (u) -- (v1);
    \draw[dotted, line width=0.3mm] (v15) -- (v);

   }   
\end{tikzpicture}
    \caption{Structure of a $Q_i^P$ with 15 vertices in $Q_i$. The dashed line represents the subpaths that are connected to $Q_i^P$ on both ends.}
    \label{fig:proper_sub}
\end{figure}

In addition, we have the following two lemmas based on Construction~\ref{cons:proper} which imply Theorem~\ref{thm:ProperNPC}.

\begin{lemma}
    \label{lem:properinteval_onlyif}
    Let $X$ be an instance of the \DP{}. Let $G_P$ be the graph obtained from $X$ using Construction~\ref{cons:proper}. If the burning number of $G_P$ is $2m+1$, where the $m^{th}$ odd number is the maximum element in $X$, then there exists a partition of $X$ into triples such that each triple sums to $B$.
\end{lemma}
\begin{proof}
    Assume that $b(G_P)=2m+1$ and $B_s=(b_1,b_2,\dots,b_{2m+1})$ is an optimal burning sequence for $G_P$. First, we have the following claims.
    \renewcommand\qed{$\hfill\square$}

    \begin{claim}
        \label{clm:burn_source_onpath}
        Each burning source in $B_s$ should be on the path $P$.
    \end{claim}
    
    \begin{proof}
        Suppose, for contradiction, that $b_i$ in $B_s$ lies on $Q_j^P - Q_j$ for some $i \in [2m+1]$ and $j \in [m+1]$. Then the subgraph induced by the vertices in $(N_{G_P}^{2m+1-i}[b_i]) \cap V(P)$ contains less than $2(2m+1-i)+1$ vertices. Then, by Observation~\ref{obs:PathOptBurn}, the subgraph induced by $\bigcup_{i=1}^{2m+1}[(N_{G_P}^{2m+1-i}[b_i]) \cap V(P)]$ has less than $(2m+1)^2$ vertices, implying that the subgraph $P$ of $G_P$ is not completely burned, a contradiction since $B_s$ is a burning sequence for $G_P$.
    \end{proof}

    \begin{claim}
        \label{clm:OneSourceInQ_i}
        For all $i \in [m+1]$, the subpath $Q_i$ contains exactly one burning source. Moreover, that source precisely is $c(Q_i)$, i.e., $b_i = c(Q_i)$, where $c(Q_i)$ is the middle vertex of $Q_i$.
    \end{claim}

    \begin{proof}
        Assume, for contradiction, that the subpath $Q_i$ contains at least two burning sources for some $i \in [m+1]$. By Claim~\ref{clm:burn_source_onpath}, all the burning sources in $B_s$ are on the path $P$. Then, by Lemma~\ref{lem:bcluster}, the burning clusters corresponding to at least two of these burning sources have a nonempty intersection within $Q_i$. Hence, it follows from Observation~\ref{obs:PathOptBurn} that $b(G_P) > 2m+1$, a contradiction. Thus, $Q_i$ contains exactly one burning source, as desired.
        
        Furthermore, since $c(Q_i)$ is the only vertex of $Q_i$ that is at a distance at most $2m+1-i$ from both endpoints of $Q_i$, and by Observation~\ref{obs:optimalburn_proper_path}, we have $b_i = c(Q_i)$, as claimed.
    \end{proof}
    \renewcommand\qed{$\hfill\blacksquare$}
    
    Now, consider $P'= G_P- \cup_{i=1}^{m+1} V(Q_i^P)$. Note that $P'$ is a disjoint union of paths $S_1,S_2, \dots , S_n$ and $S'_1,S'_2, \dots , S'_k$. By Claim~\ref{clm:OneSourceInQ_i}, the burning sequence $B_s' = (b_{m+2},b_{m+3}\ldots, b_{2m+1})$ of length $m$, should burn $P'$. Since $P'$ is a path forest of order $m^2$, this implies that for each $i \in \{m+2,m+3, \dots , 2m+1\}$, the subgraph induced by the vertices in $N_{G_P}^{2m+1-i}[b_i]$ is a path of order $2(2m+1-i)+1$, as $B_s$ is an optimal burning sequence. In other words, the corresponding burning clusters are the paths of order $2m-1,2m-3, \dots , 3,1$. Hence, there exists a partition of $P'$ induced by $B_s'$, with a one-to-one correspondence between the orders of the subpaths in the partition and the elements of $Z$.

    Note that each subpath $S'_i$, for $1\le i \le k$, should contain at least one burning source. Hence, all $S'_i$'s together require at least $k$ burning sources. Since $k = m-3n$, we are left with at most $3n$ burning sources for burning the subpaths $S_j$'s, for $1 \le j \le n$. Note that each $|S_j|=B$, which is odd. Moreover, each burning cluster corresponding to a vertex of $P'$ has size at most $2m-1$. Since $2m-1 < \frac{B}{2}$, the combined size of any two such clusters is strictly less than $B$. Consequently, covering all $B$ vertices of each subpath $S_j$ requires at least three distinct burning clusters. As at most $3n$ burning sources remain, the pigeonhole principle implies that each $S_j$, for $1 \le j \le n$, contains exactly three burning sources. It follows that each $S'_i$, for $1 \le i \le k$, contains exactly one burning source. Since the burning clusters in $P'$ form a partition whose sizes are precisely the elements of $Z=\{1,3,\dots,2m-1\}$, and since the subpaths $S'_i$ correspond exactly to the elements of $Y=Z\setminus X$, the remaining burning clusters assigned to the subpaths $S_j$ must have sizes given precisely by the elements of $X$. Therefore, for each $S_j$, the three burning clusters covering it have sizes equal to the distinct elements of $X$ whose sum is $B$. This establishes a distinct 3-partition of $X$, as desired.
\end{proof}

\begin{lemma}
    \label{lem:properinteval_if}
    Let $X$ be an instance of the \DP{}. Let $G_P$ be the graph obtained from $X$ using Construction~\ref{cons:proper}. If there exists a partition of $X$ into triples such that each triple sums to $B$, then the burning number of $G_P$ is $2m+1$, where the $m^{th}$ odd number is the maximum element in $X$.
\end{lemma}

\begin{proof}
    Assume that there exists a partition of $X$ into triples, $T_1,T_2, \dots , T_n$ (say, each $T_i = \{t_i^1,t_i^2,t_i^3\}$) such that each triple sums to $B$, i.e., for each $i \in [n]$, we have $t_i^1 + t_i^2 + t_i^3 = B$. We also have $P= (\bigcup_{i=1}^nS_i)\cup (\bigcup_{j=1}^kS_j') \cup (\bigcup_{l=1}^{m+1}Q_l)$. Note that for each $i \in [n]$, since the subpath $S_i$ is of order $B$, we can partition $S_i$ into three subpaths $P_{t_i^1}, P_{t_i^2}, P_{t_i^3}$ (note that $P_r$ is a path on $r$ vertices). It is easy to see that in this partition of the subpaths $S_1,S_2,\dots,S_n$ of $P$, there is a one-to-one correspondence between the orders of the subpaths in the partition and the elements of $X$. Moreover, since the orders of the $Q_j$'s are all the odd numbers from $2m+1$ to $4m+1$ and $Z=X \cup Y$, we now have a partition of $P$ with a one-to-one correspondence between the orders of the subpaths in the partition and the first $2m+1$ odd numbers. Let $R_i$ be the $i^{th}$ largest subpath of this partition of $P$. Then, one can see that $V(P)=\bigcup_{i=1}^{2m+1}N_P^{2m+1-i}[c(R_i)]$, which implies by Observation~\ref{obs:optimalburn_proper_path} that $V(G_P)=\bigcup_{i=1}^{2m+1}N_{G_P}^{2m+1-i}[c(R_i)]$. Thus, by Observation~\ref{obs:G_PBurningLowerBound}, we have $b(G_P)= 2m+1$, as desired.
\end{proof}

Since the \DP{} is \NPC{}, we then have the following theorem due to Lemmas~\ref{lem:properinteval_onlyif} and~\ref{lem:properinteval_if}. 

\begin{theorem}\label{thm:ProperNPC}
   The \BNP\ is \NPC\ for connected proper interval graphs.
\end{theorem}

\noindent\textbf{Complexity Implication for Line Graphs:} The above reduction has an interesting implication on the complexity of the \BNP\ on the class of line graphs. Let $\mathcal{G}$ be the class of caterpillars with degree at most $3$.
Let $\mathcal{G}_L$ be the class of the line graphs of graphs in $\mathcal{G}$.  
One can see that each connected proper interval graph $G_P$ of order $q$ obtained using Construction~\ref{cons:proper} belongs to $\mathcal{G}_L$ ($G_P \in \mathcal{G}_L$); i.e., $G_P$ corresponds to the line graph of a caterpillar $G \in \mathcal{G}$ of order $q+1$ (in which the diametral path of $G$ has exactly one more vertex than the central path $P$ of $G_P$), and for each $Q_\ell^P$ in $G_P$, the edges corresponding to each vertex $q_{ij}$ are due to a corresponding leaf edge in the caterpillar $G$, being incident on a unique vertex in the stem of $G$. Thus, we have the following corollary of Theorem~\ref{thm:ProperNPC}.

\begin{corollary}\label{cor:Linecaterpillar}
    The \BNP\ is {\NPC} on line graphs of caterpillars with degree at most $3$.
\end{corollary}

\section{$P_k$-free Graphs}
A graph $G$ is said to be $P_k$-free for some positive integer $k$ if $G$ does not contain $P_k$ (a path on $k$ vertices) as an induced subgraph. Recall that paths and cycles are two graph classes for which the burning number attains the maximum value proposed by Conjecture~\ref{conj:BurningNumber}. This motivated us to explore the burning on graphs that do not contain paths of \emph{``specific length''} as induced subgraphs, i.e., $P_k$-free graphs, where $k \ge 2$. Theorem~\ref{thm:PkFreeBound} gives an upper bound for the burning number of connected $P_k$-free graphs and, thereby, settles Conjecture~\ref{conj:BurningNumber} affirmatively for $P_k$-free graphs for any integer $k < 2\ceil{\sqrt{n}~} - 1$. We first note the following structural result on connected $P_k$-free graphs from~\cite{Camby2016P_kFreeChar}, and then utilize it to prove Theorem~\ref{thm:PkFreeBound}.

\begin{theorem}[\cite{Camby2016P_kFreeChar}]\label{thm:PkFreeStructureCamby}
    Let $G$ be a connected $P_k$-free graph of order $n$ with $k \ge 4$, and let $D$ be a minimum connected dominating set of $G$. Then $G[D]$ is either $P_{k-2}$-free or isomorphic to $P_{k-2}$. Further, such a connected dominating set can be obtained in $O(n^7)$ time.
\end{theorem}

\begin{theorem}\label{thm:PkFreeBound}
    Let $G$ be a connected $P_k$-free graph with $k \ge 2$. Then, $b(G) \le \ceil*{\frac{k+1}{2}}$.
\end{theorem}

\begin{proof}
    The proof is based on induction on $k$. Note that the result is trivial for $k=2$ and $k=3$, which form the base cases. Let $r\ge 4$ be a positive integer, and for any connected $P_k$-free graph $H$ with $2 \le k \le r-1$, we have $b(H) \le \ceil{\frac{k+1}{2}}$. Now, let $G$ be a connected $P_r$-free graph. Since $r\ge 4$, by Theorem~\ref{thm:PkFreeStructureCamby}, we have a connected dominating set $D$ of $G$ such that $G[D]$ is either a $P_{r-2}$-free graph or a graph isomorphic to $P_{r-2}$. Further, since $D$ is a dominating set of $G$, we have $b(G) \le b(G[D]) + 1$.
    
    If $G[D]$ is a connected $P_{r-2}$-free graph, then by the induction hypothesis, we have $b(G[D]) \le \ceil{\frac{r-1}{2}}$. This implies that $b(G)\le \ceil{\frac{r+1}{2}}$, and we are done. On the other hand, if  $G[D]$ is isomorphic to $P_{r-2}$, then by Theorem~\ref{thm:PathAttainsTheBound}, we have that $b(G[D])=b(P_{r-2}) = \ceil{\sqrt{r-2}~}$. Since $r \ge 4$, we then have $b(G)\le \ceil{\sqrt{r-2}}+1\le \ceil{\frac{r+1}{2}}$. Hence, the theorem.
\end{proof}

\noindent\textbf{Tightness of Theorem~\ref{thm:PkFreeBound}:}
We claim that the upper bound in Theorem~\ref{thm:PkFreeBound} is tight. To justify this, consider the graph $\Tilde{G}$ shown in Fig.~\ref{fig:P6Free} which is a $P_6$-free graph with $b(\Tilde{G}) = 4 = \ceil*{\frac{k+1}{2}}$.  We also show that for any integer $k$ with $2 \le k < 2\lceil \sqrt{n}~\rceil - 1$, the upper bound in Theorem~\ref{thm:PkFreeBound} is tight up to an additive constant 1. For this, we present an infinite subclass of $P_k$-free graphs having their burning number exactly equal to $\floor*{\frac{k}{2}}$. Let $\mathbb{P}_r$ be a spider with the degree of the branching vertex being $r$, and the length of each leg (a total of $r$ legs) being $r-1$. The graph $\mathbb{P}_r$ is shown in Fig.~\ref{fig:PrFree}. One can verify that if $r = \floor*{\frac{k}{2}}$, then $\mathbb{P}_r$ is a $P_k$-free graph with $b(\mathbb{P}_r) = r$.

\begin{figure}[ht]
    \begin{subfigure}[b]{0.3\linewidth}
        \centering
        \begin{tikzpicture}[scale=0.75]
    \begin{scope}[every node/.style={empty vertex}]
        \node [label={[above]1:{$x_1$}}] (x1) at (1,3) {};
        \node [label={[above]1:{$x_2$}}] (x2) at (1,2) {};
        \node [label={[above]1:{$x_3$}}] (x3) at (1,1) {};
        \node [label={[above]1:{$x_4$}}] (x4) at (1,0) {};
        \node [label={[above]1:{$x'_1$}}] (x'1) at (0,3) {};
        \node [label={[above]1:{$x'_2$}}] (x'2) at (0,2) {};
        \node [label={[above]1:{$x'_3$}}] (x'3) at (0,1) {};
        \node [label={[above]1:{$x'_4$}}] (x'4) at (0,0) {};
        \node [label={[above]1:{$y_1$}}] (y1) at (2.5,3) {};
        \node [label={[above]1:{$y_2$}}] (y2) at (2.5,2) {};
        \node [label={[above]1:{$y_3$}}] (y3) at (2.5,1) {};
        \node [label={[above]1:{$y_4$}}] (y4) at (2.5,0) {};
        \node [label={[above]1:{$y'_1$}}] (y'1) at (3.5,3) {};
        \node [label={[above]1:{$y'_2$}}] (y'2) at (3.5,2) {};
        \node [label={[above]1:{$y'_3$}}] (y'3) at (3.5,1) {};
        \node [label={[above]1:{$y'_4$}}] (y'4) at (3.5,0) {};
    \end{scope}
    \draw (x1) to (y1);
    \draw (x1) to (y2);
    \draw (x1) to (y3);
    \draw (x1) to (y4);
    \draw (x2) to (y1);
    \draw (x2) to (y2);
    \draw (x2) to (y3);
    \draw (x2) to (y4);
    \draw (x3) to (y1);
    \draw (x3) to (y2);
    \draw (x3) to (y3);
    \draw (x3) to (y4);
    \draw (x4) to (y1);
    \draw (x4) to (y2);
    \draw (x4) to (y3);
    \draw (x4) to (y4);
    \draw (x1) to (x'1);
    \draw (x2) to (x'2);
    \draw (x3) to (x'3);
    \draw (x4) to (x'4);
    \draw (y1) to (y'1);
    \draw (y2) to (y'2);
    \draw (y3) to (y'3);
    \draw (y4) to (y'4);
\end{tikzpicture}
        \caption{Graph $\Tilde{G}$}
        \label{fig:P6Free}
    \end{subfigure}
    \begin{subfigure}[b]{0.69\linewidth}
        \centering
        \begin{tikzpicture}[scale=0.75]
    \begin{scope}[every node/.style={empty vertex}]
    \node [label=below:{$x$}] (x) at (0,1.8) {};
    \node [label=below:{$x_1^1$}] (x11) at (-1,1.8) {};
    \node [label=below:{$x_1^2$}] (x12) at (-2,1.8) {};
    \node [label={[below]1:{$x_1^{r-1}$}}] (x1r-1) at (-4.25,1.8) {};
    \node [label=below:{$x_2^1$}] (x21) at (-1,0) {};
    \node [label=below:{$x_2^2$}] (x22) at (-2,0) {};
    \node [label={[below]1:{$x_2^{r-1}$}}] (x2r-1) at (-4.25,0) {};
    \node [label={[below]1:{$x_{r-1}^1$}}] (xr-11) at (1,0) {};
    \node [label={[below]1:{$x_{r-1}^2$}}] (xr-12) at (2,0) {};
    \node [label={[below]1:{$x_{r-1}^{r-1}$}}] (xr-1r-1) at (4.25,0) {};
    \node [label=below:{$x_r^1$}] (xr1) at (1,1.8) {};
    \node [label=below:{$x_r^2$}] (xr2) at (2,1.8) {};
    \node [label={[below]1:{$x_r^{r-1}$}}] (xrr-1) at (4.25,1.8){};
    \end{scope}
    
    \node (y) at (-0.6,0.6) {};
    \node (z) at (0.6,0.6) {};
    \node (x13) at (-2.75,1.8) {};
    \node (x14) at (-3.25,1.8) {};
    \node (x23) at (-2.75,0) {};
    \node (x24) at (-3.25,0) {};
    \node (xr-13) at (2.75,0) {};
    \node (xr-14) at (3.25,0) {};
    \node (xr3) at (2.75,1.8) {};
    \node (xr4) at (3.25,1.8) {};
    
    \draw (x) to (x11);
    \draw (x11) to (x12);
    \draw (x12) to (x13);
    \draw (x14) to (x1r-1);
    \draw [thick, densely dotted] (x13) to (x14);
    \draw (x) to (x21);
    \draw (x21) to (x22);
    \draw (x22) to (x23);
    \draw (x24) to (x2r-1);
    \draw [thick, densely dotted] (x23) to (x24);
    \draw (x) to (xr-11);
    \draw (xr-11) to (xr-12);
    \draw (xr-12) to (xr-13);
    \draw (xr-14) to (xr-1r-1);
    \draw [thick, densely dotted] (xr-13) to (xr-14);
    \draw (x) to (xr1);
    \draw (xr1) to (xr2);
    \draw (xr2) to (xr3);
    \draw (xr4) to (xrr-1);
    \draw [thick, densely dotted] (xr3) to (xr4);
    \draw [thick, loosely dotted, bend right] (y) to (z);
\end{tikzpicture}
        \caption{Graph $\mathbb{P}_r$}
        \label{fig:PrFree}
    \end{subfigure}
    \caption{Tight examples for Theorem~\ref{thm:PkFreeBound}}
    \label{fig:PkFree}
\end{figure}
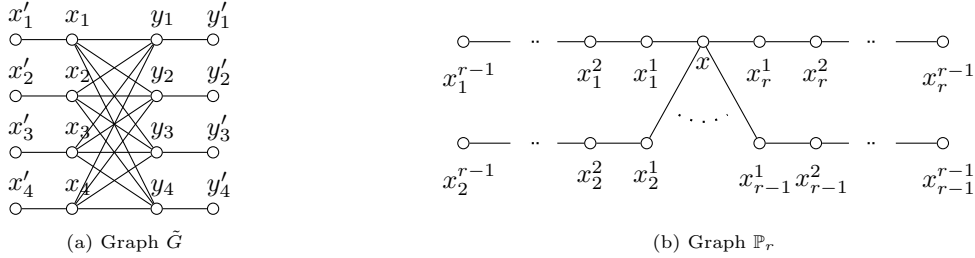

\smallskip

\noindent\textbf{Algorithmic Consequences:} Our simple combinatorial proof of Theorem~\ref{thm:PkFreeBound} also has an algorithmic consequence. Using the trivial bound $b(G) \le k-1$ for $P_k$-free graphs and a brute-force approach, it was shown in~\cite{Kamali2020BurningTwoWorlds} that an optimal burning sequence of $G$ can be found in $O(n^{k+1})$ time. For a constant value of $k$, our improved bound $b(G) \le \ceil*{\frac{k+1}{2}}$ for $P_k$-free graphs (Theorem~\ref{thm:PkFreeBound}) leads to a better running time. Indeed, since it suffices to enumerate burning sequences of length at most $\ceil*{\frac{k+1}{2}}$, and there are $O(n^{\ceil*{\frac{k+1}{2}}})$ such sequences, while verifying whether a sequence is a valid burning sequence takes $O(n^2)$ time, an optimal burning sequence of a $P_k$-free graph can be found in $O(n^{\ceil*{\frac{k+1}{2}}+2})$ time.

On the other hand, when $k$ is not a constant, a burning sequence of size at most $\ceil*{\frac{k+1}{2}}$ can still be constructed in polynomial time. Indeed, by Theorem~\ref{thm:PkFreeStructureCamby}, a minimum connected dominating set $D$ of a connected $P_k$-free graph can be computed in $O(n^7)$ time such that $G[D]$ is either $P_{k-2}$-free or isomorphic to $P_{k-2}$. The proof of Theorem~\ref{thm:PkFreeBound} is constructive; it recursively computes a burning sequence for $G[D]$ and extends it to a burning sequence for $G$. Since the recursion depth is at most $\ceil*{\frac{k+1}{2}}$, the overall running time is $O(kn^7)$. Therefore, we have the following corollary.


\begin{corollary}
    Let $G$ be a connected $P_k$-free graph with $k\geq 2$. A burning sequence of $G$ having size at most $\ceil*{\frac{k+1}{2}}$ can be found in $O(kn^7)$ time.
\end{corollary}

Note that for any integer $k < 2\ceil{\sqrt{n}~} - 1$, Theorem~\ref{thm:PkFreeBound} settles Conjecture~\ref{conj:BurningNumber} affirmatively for $P_k$-free graphs of order $n$. Moreover, we showed that the bound in Theorem~\ref{thm:PkFreeBound} is tight when $k=6$, and the bound is tight up to an additive constant $1$ for any $k$ with $2 \le k < 2\ceil{\sqrt{n}~} - 1$.

\section{Variants of Graph Burning}
In this section, we study two variants of graph burning, namely edge burning and total burning. Note that, for a graph $G$, the burning number may differ from its edge burning number and total burning number. This naturally motivates the study of these variants.

\subsection{Edge Burning}
As the name ``\emph{edge burning}'' suggests, here we burn only the edges of the graph. At each step $i$, a new edge is burned along with the unburned edges that are neighbors of the edges burned before the $i^{th}$ step. The \emph{edge burning number} of a graph $G$, denoted by $b_L(G)$, is the minimum number of steps required to burn all the edges of $G$. The \emph{line graph} of a graph $G$, denoted by $L(G)$, is the graph with the vertex set $E(G)$, and $xy \in E(L(G))$ if and only if the corresponding edges $x$ and $y$ are neighbors in $G$. It is easy to see that $b_L(G) = b(L(G))$ for any graph $G$. Therefore, computing $b_L(G)$ is equivalent to computing $b(L(G))$. Henceforth, we use $b_L(G)$ to denote the burning number of the line graph $L(G)$. Since the line graph of a path $P_n$ is $P_{n-1}$ and the line graph of a cycle $C_n$ is $C_n$ itself, we have the following Observation due to Theorem~\ref{thm:PathAttainsTheBound}.

\begin{observation}
    For a path $P_n$ and a cycle $C_n$, $b_L(P_n) = \ceil{\sqrt{n-1}~}$ and $b_L(C_n) = \ceil{\sqrt{n}~}$.
\end{observation}

First, we present a couple of Lemmas with the intention of proving Theorem~\ref{thm:LineGraphBounds}.

\begin{lemma}
    \label{lem:LineGraphBounds_first}
    For any graph $G$, $b(G) \le b_L(G) + 1$.
\end{lemma}

\begin{proof}
    Let $G$ be a graph with $n$ vertices and $m$ edges, and let $H = L(G)$, with $V(H) = E(G) =\{e_1, e_2, \dots, e_m\}$. Let $b_L(G) = k$ and let $B_L = (b_1, b_2, \dots, b_k)$ be an optimal burning sequence for $H$. Now, consider the sequence $B = (a_1, a_2, \dots , a_k, a_{k+1})$ in $G$, where $a_i$ is an arbitrary end vertex of the edge $b_i$ in $G$ for $1 \le i \le k$, and $a_{k+1}$ is an unburned vertex after the $k^{th}$ iteration (if it exists) in $G$, chosen arbitrarily. Now, it is enough to show that $B$ is a burning sequence for $G$, since $|B| = b_L(G)+1$.

    Recall that for $1\le i\le |B|$, $S_i(B)$ and $S_i^-(B)$ are the sets of vertices of $G$ burned in the $i^{th}$ step and within the first $i$ steps, respectively. Let $e_i$ be an arbitrary vertex in $H$ for some $i \in [m]$, and let $e_i=uv$ be the corresponding edge in $G$. Now, we ``claim'' that if $e_i \in S_l(B_L)$ in $H$ for some $l \in [k]$, then $u \in S_l^-(B)$ or $v \in S_l^-(B)$ in $G$. Then, since all vertices of $H$ are burned in $k$ steps by $B_L$ and $e_i$ was chosen arbitrarily, all vertices of $G$ are burned in $k+1$ steps by $B$. This implies that $B$ is a burning sequence for $G$, as desired. Thus, it remains to prove our ``claim''.

    We prove this claim by induction on $l$. The base case, when $l=1$, is trivial. Let the hypothesis be true for all $l \le j-1$. We need to prove the hypothesis for $l=j$. We have a vertex $e_i$ of $H$ with $e_i \in S_j(B_L)$, which corresponds to the edge $uv$ in $G$. This can happen in two ways: (i) $e_i\notin B_L$ or (ii) $e_i\in B_L$. Suppose (i) is true. Then the vertex $e_i$ is burned by a fire spread from its neighbor, say $e_p$, which is burned in the $(j-1)^{th}$ iteration. Since $e_p e_i \in E(H)$, either $e_p = uw$ or $e_p = vw$ in $G$ for some $w \in V(G)$. In either case, by the induction hypothesis, at least one vertex in $\{u, v, w\}$ is in $S_{j-1}^-(B)$. Hence, at least one vertex in $\{u,v\}$ is in $S_j^-(B)$. 
    On the other hand, suppose (ii) is true. Then $b_j = e_i$, i.e., the vertex $e_i$ is the $j^{th}$ burning source in $B_L$. Hence, by choice of $B$, one of the end vertices of $e_i$, say $u$, will be the $j^{th}$ burning source in $B$, implying that $u \in S_l^-(B)$, which completes the proof.
\end{proof}

\begin{lemma}
    \label{lem:LineGraphBounds_second}
     For any graph $G$, $b_L(G) \le b(G) + 1$.  
\end{lemma}

\begin{proof}
    Let $G$ be a graph with $n$ vertices and $m$ edges, and let $H = L(G)$, with $V(H) = E(G)$. Let $b(G) = k$ and $B = (b_1, b_2, \dots, b_k)$ be an optimal burning sequence for $G$. Consider the sequence $B_L = (a_1, a_2, \ldots, a_k, a_{k+1})$ in $H$, where for every $i \in [k]$, $a_i$ in $V(H)$ corresponds to any one edge incident on the vertex $b_i$ in $G$, and $a_{k+1}$ is a vertex of $H$, unburned in the first $k$ steps, chosen arbitrarily. It is enough to prove that $B_L$ is a burning sequence for $H$, since $|B_L| = b(G)+1$.
    
    Let $v_i$ be an arbitrary vertex in $G$ for some $i \in [n]$. Now, we ``claim'' that if $v_i \in S_l(B)$ in $G$ for some $l \in [k]$, then at least one vertex in $H$ corresponding to the edges incident on $v_i$ in $G$ belongs to $S_l^-(B_L)$. Then, since all vertices of $G$ are burned in $k$ steps by $B$ and $v_i$ was chosen arbitrarily, all vertices of $H$ (edges of $G$) are burned in $k+1$ steps by $B_L$. This implies that $B_L$ is a burning sequence for $H$, as desired. Thus, it remains to prove our ``claim''.

    We prove this claim by induction on $l$. The base case, when $l=1$, is trivial. Let the hypothesis be true for all $l \le j-1$. We need to prove the hypothesis for $l=j$. We have a vertex $v_i$ of $G$ with $v_i \in S_j(B)$. This can happen in two ways: (i)  $v_i \notin B$  or (ii) $v_i \in B$. Suppose (i) is true. Then the vertex $v_i$ gets burned by a fire spread from one of its neighbors, say $v_p$, which is burned in the $(j-1)^{th}$ iteration. By induction hypothesis, the vertex in $H$ corresponding to the edges incident on $v_p$ in $G$, is burned in the $(j-1)^{th}$ iteration in $H$ by $B_L$, which implies that the vertex in $H$ corresponding to the edge $v_i v_p$ in $G$, belongs to $S_j^-(B_L)$. 
    On the other hand, suppose (ii) is true. Then $v_i$ must be the $j^{th}$ burning source in $B$, i.e., $b_j = v_i$. Since the vertex corresponding to an edge incident on $v_i$ in $G$ is selected as the $j^{th}$ burning source in $B_L$, the claim is true for $l=j$. This completes the proof.
\end{proof}

Now, Theorem~\ref{thm:LineGraphBounds} that relates $b_L(G)$ and $b(G)$, is immediate from Lemma~\ref{lem:LineGraphBounds_first} and Lemma~\ref{lem:LineGraphBounds_second}.

\begin{theorem}\label{thm:LineGraphBounds}
   For a graph $G$,  $b(G) - 1 \le b_L(G) \le b(G) + 1$.
\end{theorem}

Note that the lower and the upper bounds in Theorem~\ref{thm:LineGraphBounds} are tight. For instance, as $L(P_k)$ is isomorphic to  $P_{k-1}$, for any $k=n^2+1$ with $n\ge 1$, we have $b(P_k) = n+1 $ and  $b_L(P_k) = n$. Similarly, one can verify that $b(K_n) = 2$ and  $b_L(K_n) = 3$ for every $n\ge 5$. But when we restrict ourselves to trees, we have a stronger upper bound as stated in the following theorem.

\begin{theorem}\label{thm:LineGraphTreeBounds}
    For a tree $T$, $b_L(T) \le b(T)$.
\end{theorem}

\begin{proof}
    Let $T$ be a tree with $b(T) = k$. We assume that $T$ is a rooted tree with an arbitrary vertex $x$ being its root. Let $T_1$ be the line graph of $T$. Note that the vertices corresponding to the set of all edges incident on a particular vertex in $T$ create a block in $T_1$. Let $t_x$ be the block in $T_1$ corresponding to the edges incident on the root vertex $x$. Let $T_2$ be a graph obtained from $T_1$ by adding a vertex $y$ and making it adjacent to every vertex in the block $t_x$ (see Fig.~\ref{fig:ModifiedLineGraph}). Since the graph induced by the vertices in $N_{T_2}^1[y]$ is a block in $T_2$, we have that for any neighbor $y'$ of $y$ and any positive integer $r$, $N_{T_2}^r[y] \cap V(T_1) \subseteq N_{T_1}^r [y']$. Hence, $T_1$ is a good-monotonic subgraph of $T_2$, which implies that $b(T_1) \le b(T_2)$. An example of a tree $T$ with the corresponding graphs $T_1$ and $T_2$ is provided in Fig.~\ref{fig:TreeWithT1T2}. Note that in this example, $t_x$ is the induced subgraph $T_1[\{e_1,e_2\}]$.

    \begin{figure}[ht]
    \begin{subfigure}[b]{0.33\linewidth}
        \centering
        \begin{tikzpicture}[scale=0.7]
    \begin{scope}[every node/.style={empty vertex,inner sep=0.5pt,minimum size=3.5mm}]
        \node[label= above: \small${x}$] (v1) at (0,4.5){\tiny{$v_0$}};
        \node (v2) at (-1.5,3) {\tiny{$v_1$}};
        \node (v3) at (1.5,3) {\tiny{$v_2$}};
        \node (v4) at (-2.5,1.5) {\tiny{$v_3$}};
        \node (v5) at (-0.5,1.5) {\tiny{$v_4$}};
        \node (v6) at (0.5,1.5) {\tiny{$v_5$}};
        \node (v7) at (1.5,1.5) {\tiny{$v_6$}};
        \node (v8) at (2.5,1.5) {\tiny{$v_7$}};
        \node (v9) at (-0.5,0) {\tiny{$v_8$}};
        \node (v10) at (1.5,0) {\tiny{$v_9$}};
    \end{scope}
    \draw (v1) to node[below,label={[left]\tiny{$e_1$}}] {} (v2);
    \draw (v1) to node[below,label={[right]\tiny{$e_2$}}] {} (v3);
    \draw (v2) to node[below,label={[left]\tiny{$e_3$}}] {} (v4);
    \draw (v2) to node[below,label={[right]\tiny{$e_4$}}] {} (v5);
    \draw (v3) to node[below,label={[left]\tiny{$e_5$}}] {} (v6);
    \draw (v3) to node[below,label={[below]\tiny{$e_6$}}] {} (v7);
    \draw (v3) to node[below,label={[right]\tiny{$e_7$}}] {} (v8);
    \draw (v6) to node[below,label={[left]\tiny{$e_8$}}] {} (v9);
    \draw (v6) to node[below,label={[right]\tiny{$e_9$}}] {} (v10);
\end{tikzpicture}
        \caption{A tree $T$ ($\cong T_3$)}
        \label{fig:SampleTree}
    \end{subfigure}
    \begin{subfigure}[b]{0.33\linewidth}
        \centering
        \begin{tikzpicture}[scale=0.7]
    \begin{scope}[every node/.style={empty vertex,inner sep=0.5pt,minimum size=3.5mm}]
        \node (v1) at (-1.5,3) {\tiny{$e_1$}};
        \node (v2) at (1.5,3) {\tiny{$e_2$}};
        \node (v3) at (-2.5,1.5) {\tiny{$e_3$}};
        \node (v4) at (-0.5,1.5) {\tiny{$e_4$}};
        \node (v5) at (0.5,1.5) {\tiny{$e_5$}};
        \node (v6) at (1.5,1.5) {\tiny{$e_6$}};
        \node (v7) at (2.5,1.5) {\tiny{$e_7$}};
        \node (v8) at (-0.5,0) {\tiny{$e_8$}};
        \node (v9) at (1.5,0) {\tiny{$e_9$}};
    \end{scope}
    \draw (v1) to (v2);
    \draw (v1) to (v3);
    \draw (v1) to (v4);
    \draw (v2) to (v5);
    \draw (v2) to (v6);
    \draw (v2) to (v7);
    \draw (v3) to (v4);
    \draw (v5) to (v6);
    \draw [bend left] (v5) to (v7);
    \draw (v5) to (v8);
    \draw (v5) to (v9);
    \draw (v6) to (v7);
    \draw (v8) to (v9);
\end{tikzpicture}
        \caption{Line graph $T_1$ of $T$}
        \label{fig:LineGraphOfTree}
    \end{subfigure}
    \begin{subfigure}[b]{0.33\linewidth}
        \centering
        \begin{tikzpicture}[scale=0.7]
    \begin{scope}[every node/.style={empty vertex,inner sep=0.5pt,minimum size=3.5mm}]
        \node (y) at (0,4.5) {\tiny{$y$}};
        \node (v1) at (-1.5,3) {\tiny{$e_1$}};
        \node (v2) at (1.5,3) {\tiny{$e_2$}};
        \node (v3) at (-2.5,1.5) {\tiny{$e_3$}};
        \node (v4) at (-0.5,1.5) {\tiny{$e_4$}};
        \node (v5) at (0.5,1.5) {\tiny{$e_5$}};
        \node (v6) at (1.5,1.5) {\tiny{$e_6$}};
        \node (v7) at (2.5,1.5) {\tiny{$e_7$}};
        \node (v8) at (-0.5,0) {\tiny{$e_8$}};
        \node (v9) at (1.5,0) {\tiny{$e_9$}};
    \end{scope}
    \draw (v1) to (v2);
    \draw (v1) to (v3);
    \draw (v1) to (v4);
    \draw (v2) to (v5);
    \draw (v2) to (v6);
    \draw (v2) to (v7);
    \draw (v3) to (v4);
    \draw (v5) to (v6);
    \draw [bend left] (v5) to (v7);
    \draw (v5) to (v8);
    \draw (v5) to (v9);
    \draw (v6) to (v7);
    \draw (v8) to (v9);
    \draw (y) to (v1);
    \draw (y) to (v2);
\end{tikzpicture}
        \caption{Graph $T_2$ obtained from $T_1$}
        \label{fig:ModifiedLineGraph}
    \end{subfigure}
    \caption{An example of a tree with the corresponding graphs $T_1$ and $T_2$.}
    \label{fig:TreeWithT1T2}
    \end{figure}
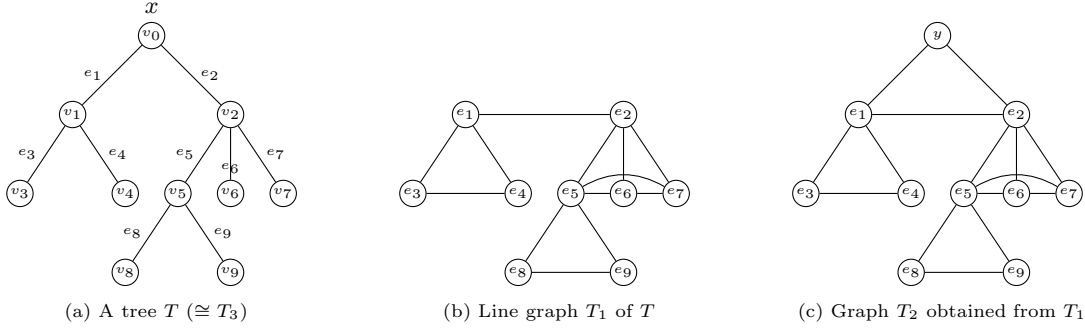
    
Now, we obtain a tree $T_3$ by applying the Breadth-First Search on $T_2$ with $y$ as the source vertex. We claim $T_3 \cong T$. To prove this, we explicitly provide a bijection $f : V(T) \rightarrow V(T_3)$. First, we assign $f(x) = y$, where $x$ is the root of $T$. Note that $T_3$ does not contain any edges between the vertices that are at the same level in the BFS order of $T_2$. For any vertex $z$ in $V(T) \setminus \{x\}$, there exists a unique edge $e_z$ between $z$ and its parent in $T$. Since $T_2$ is obtained from the line graph of $T$, by the definition of $T_3$, there exists a vertex $v_{e_z}$ in $T_3$ corresponding to the edge $e_z$ in $T$. Then for any vertex $z$ in $V(T) \setminus \{x\}$, we assign $f(z) = v_{e_z}$. One can verify that $f$ is the desired bijection, implying that $b(T_3) = b(T)$. Since edge deletion in a graph does not decrease its burning number, we have $b(T_2) \le b(T_3)$. Hence, we collectively have the following $b_L(T) = b(T_1) \le b(T_2) \le b(T_3) = b(T) = k$, as desired.
\end{proof}

\subsection{Total Burning}
In the ``\emph{total burning}'' problem, we burn both vertices and edges of the graph. In each step $i$, an unburned vertex or edge is burned, and the fire spreads to all vertices and edges that are neighbors of some vertex or edge that is burned within step $i-1$ until all the vertices and edges of $G$ have been burned. The \emph{total burning number} of a graph $G$, denoted by $b_T(G)$, is the minimum number of steps required to burn all vertices and edges of $G$. The \emph{total graph} of a graph $G$, denoted by $T(G)$, is the graph with the vertex set $V \cup E$, and $xy \in E(T(G))$ if and only if the corresponding $x$ and $y$ are neighbors in $G$. It is easy to see that $b_T(G) = b(T(G))$ for any graph $G$. Therefore, computing $b_T(G)$ is equivalent to computing $b(T(G))$. Henceforth, we use $b_T(G)$ to denote the burning number of a total graph $T(G)$. Moghbel~\cite{Moghbel2020TotalBurning} introduced the total burning problem and conjectured the following relationship between $b_T(G)$ and $b(G)$, which we settle positively in Theorem~\ref{thm:TotalGraphBounds}.

\begin{conjecture}[\cite{Moghbel2020TotalBurning}]\label{conj:TotalBurning}
    For a connected graph $G$ with the burning number $b(G)$, and its total graph $T(G)$ with the burning number $b_T(G)$, we have $b(G) \le b_T(G) \le b(G) + 1$.
\end{conjecture}

\begin{theorem}\label{thm:TotalGraphBounds}
   Let $G$ be a connected graph, then $b(G) \le b_T(G) \le b(G) + 1$.
\end{theorem}

\begin{proof}
    Let $H = T(G)$. Then, $V(H)$ can be partitioned into two sets as $V(H) = \mathcal{V} \uplus \mathcal{E}$ where $\mathcal{V} = \{v_1,v_2, \dots ,v_n\}$ and $\mathcal{E} = \{e_1,e_2, \dots ,e_m\}$ correspond to the vertex set and the edge set of $G$.
    
    \emph{Upper Bound:} Let $B$ be an optimal burning sequence for $G$. Now, if we burn $H$ as per $B$, any vertex of $H$ that is unburned at the end of $|B|$ steps will be in $\mathcal{E}$. But any such vertex is adjacent to some vertex in $\mathcal{V}$, and hence, one additional step is sufficient to burn all such vertices. Hence, $b_T(G) \le b(G) + 1$.
    
    \emph{Lower Bound:} Let $b(H) = k$ and $A = (a_1, a_2, \dots , a_k)$ be an optimal burning sequence for $H$. Note that for any two vertices $u,v \in V(G)$, we have $d_G(u,v) = d_{H}(u,v)$ by the definition of total graphs. For the rest of the proof, ``for all $i \in [k]$'' is implicitly assumed whenever we mention the subscript $i$. Now, we obtain a vertex sequence $B = (b_1, b_2, \dots , b_k)$ of $G$ from $A$ as follows. If $a_i \in \mathcal{V}$, then $b_i = a_i$. Otherwise, if $a_i \in \mathcal{E}$, then $b_i \in \mathcal{V}$ is any of the endpoints of the edge $a_i$ in $G$, chosen arbitrarily. Now, it remains to prove that $B$ is a burning sequence for $G$. One can see that the following inclusion follows from the choice of $b_i$ and the fact that $d_G(u,v) = d_{H}(u,v)$ for all $\{u,v\}\in V(G)$.
    \begin{equation}\label{eqn:NbrsOfBiCoverNbrsOfAi}
        N_H^{k-i}[a_i] \cap V(G) \subseteq N_G^{k-i}[b_i]
    \end{equation}
    Since $G$ is an induced subgraph of $H$ and $A$ is a burning sequence for $H$, by (\ref{eqn:NbrsOfBiCoverNbrsOfAi}) we have:
    \begin{align*}
        & V(H) = \left[N_H^{k-1}[a_1] \cup N_H^{k-2}[a_2] \dots \cup N_H^0[a_k]\right] \\
        \implies & V(G) = \left[N_H^{k-1}[a_1] \cup N_H^{k-2}[a_2] \dots \cup N_H^0[a_k]\right] \cap V(G) \\
        \implies & V(G) = \left[N_H^{k-1}[a_1] \cap V(G)\right] \cup \left[N_H^{k-2}[a_2] \cap V(G)\right] \cup \dots \cup \left[N_H^0[a_k] \cap V(G)\right] \\
        \implies & V(G) \subseteq \left[N_G^{k-1}[b_1] \cup N_G^{k-2}[b_2] \cup \dots \cup N_G^0[b_k]\right]
    \end{align*}
    Since the reverse inclusion is obvious, $B$ is a burning sequence for $G$. Therefore, we have the inequality $b(G) \le |B| = k = |A| = b_T(G)$, as desired.
\end{proof}

\noindent\textbf{Hardness of Total Burning:}

Here, we intend to prove the hardness of the total burning problem (i.e., \BNP{} on the total graphs). To prove this, we define a \emph{spike graph} of a given graph $G$, denoted by  $G_s$, and observe the relationship between the burning numbers of $G$ and the total graph $T(G_s)$ of $G_s$ (Lemma~\ref{lem:TotalSpike}). A \emph{spike graph} of a graph $G$, denoted by $G_s$, is a graph obtained from $G$ such that for each vertex $v_i\in V(G)$, for $1\leq i\leq | V(G)|$, introduce a vertex $l_i$ and add an edge between $l_i$ and $v_i$, i.e., $V(G_s)=V(G)\cup \{l_i:v_i\in V(G)\}$ and $E(G_s)=E(G)\cup\{v_il_i:v_i\in V(G)\}$. An example of a spike graph is given in Fig.~\ref{fig:SpikeC4}.

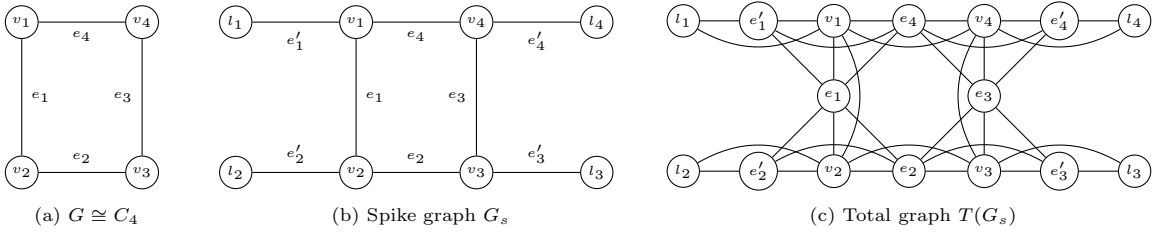
\begin{figure}[ht]
    \begin{subfigure}[b]{0.16\linewidth}
        \centering
        \begin{tikzpicture}[scale=0.8]
    \begin{scope}[every node/.style={empty vertex}]
        \node (v1) at (0,2.5) {\tiny{$v_1$}};
        \node (v2) at (0,0) {\tiny{$v_2$}};
        \node (v3) at (2,0) {\tiny{$v_3$}};
        \node (v4) at (2,2.5) {\tiny{$v_4$}};
    \end{scope}
    \draw (v1) to node[below,label={[right]\tiny{$e_1$}}] {} (v2);
    \draw (v2) to node[below,label={[above]\tiny{$e_2$}}] {} (v3);
    \draw (v3) to node[below,label={[left]\tiny{$e_3$}}] {} (v4);
    \draw (v4) to node[below,label={[below]\tiny{$e_4$}}] {} (v1);
\end{tikzpicture}
        \caption{$G \cong C_4$}
        \label{fig:C4}
    \end{subfigure}
    \begin{subfigure}[b]{0.40\linewidth}
        \centering
        \begin{tikzpicture}[scale=0.8]
    \begin{scope}[every node/.style={empty vertex}]
        \node (v1) at (0,2.5) {\tiny{$v_1$}};
        \node (v2) at (0,0) {\tiny{$v_2$}};
        \node (v3) at (2,0) {\tiny{$v_3$}};
        \node (v4) at (2,2.5) {\tiny{$v_4$}};
        \node (l1) at (-2,2.5) {\tiny{$l_1$}};
        \node (l2) at (-2,0) {\tiny{$l_2$}};
        \node (l3) at (4,0) {\tiny{$l_3$}};
        \node (l4) at (4,2.5) {\tiny{$l_4$}};
    \end{scope}
    \draw (v1) to node[below,label={[right]\tiny{$e_1$}}] {} (v2);
    \draw (v2) to node[below,label={[above]\tiny{$e_2$}}] {} (v3);
    \draw (v3) to node[below,label={[left]\tiny{$e_3$}}] {} (v4);
    \draw (v4) to node[below,label={[below]\tiny{$e_4$}}] {} (v1);
    \draw (v1) to node[below,label={[below]\tiny{$e'_1$}}] {} (l1);
    \draw (v2) to node[below,label={[above]\tiny{$e'_2$}}] {} (l2);
    \draw (v3) to node[below,label={[above]\tiny{$e'_3$}}] {} (l3);
    \draw (v4) to node[below,label={[below]\tiny{$e'_4$}}] {} (l4);
\end{tikzpicture}
        \caption{Spike graph $G_s$}
        \label{fig:SpikeC4}
    \end{subfigure}
    \begin{subfigure}[b]{0.43\linewidth}
        \centering
        \begin{tikzpicture}[scale=0.8]
    \begin{scope}[every node/.style={empty vertex}]
        \node (v1) at (0,2.5) {\tiny{$v_1$}};
        \node (v2) at (0,0) {\tiny{$v_2$}};
        \node (v3) at (2.5,0) {\tiny{$v_3$}};
        \node (v4) at (2.5,2.5) {\tiny{$v_4$}};
        \node (l1) at (-2.5,2.5) {\tiny{$l_1$}};
        \node (l2) at (-2.5,0) {\tiny{$l_2$}};
        \node (l3) at (5,0) {\tiny{$l_3$}};
        \node (l4) at (5,2.5) {\tiny{$l_4$}};
        \node (e'1) at (-1.25,2.5) {\tiny{$e'_1$}};
        \node (e'2) at (-1.25,0) {\tiny{$e'_2$}};
        \node (e'3) at (3.75,0) {\tiny{$e'_3$}};
        \node (e'4) at (3.75,2.5) {\tiny{$e'_4$}};
        \node (e1) at (0,1.25) {\tiny{$e_1$}};
        \node (e2) at (1.25,0) {\tiny{$e_2$}};
        \node (e3) at (2.5,1.25) {\tiny{$e_3$}};
        \node (e4) at (1.25,2.5) {\tiny{$e_4$}};
    \end{scope}
    \draw [bend left] (v1) to (v2);
    \draw [bend left] (v2) to (v3);
    \draw [bend left] (v3) to (v4);
    \draw [bend left] (v4) to (v1);
    \draw [bend left] (v1) to (l1);
    \draw [bend right] (v2) to (l2);
    \draw [bend left] (v3) to (l3);
    \draw [bend right] (v4) to (l4);
    \draw (v1) to (e'1);
    \draw (l1) to (e'1);
    \draw (v2) to (e'2);
    \draw (l2) to (e'2);
    \draw (v3) to (e'3);
    \draw (l3) to (e'3);
    \draw (v4) to (e'4);
    \draw (l4) to (e'4);
    \draw (v1) to (e1);
    \draw (v2) to (e1);
    \draw (v2) to (e2);
    \draw (v3) to (e2);
    \draw (v3) to (e3);
    \draw (v4) to (e3);
    \draw (v4) to (e4);
    \draw (v1) to (e4);
    \draw (e1) to (e2);
    \draw (e2) to (e3);
    \draw (e3) to (e4);
    \draw (e4) to (e1);
    \draw (e1) to (e'1);
    \draw (e1) to (e'2);
    \draw [bend right] (e2) to (e'2);
    \draw [bend left] (e2) to (e'3);
    \draw (e3) to (e'3);
    \draw (e3) to (e'4);
    \draw [bend right] (e4) to (e'4);
    \draw [bend left] (e4) to (e'1);
\end{tikzpicture}
        \caption{Total graph $T(G_s)$}
        \label{fig:TotalSpikeC4}
    \end{subfigure}
    \caption{An example of a spike graph and its total graph}
    \label{fig:SpikeExmpl}
\end{figure}

\begin{lemma}
    \label{lem:TotalSpike}
    Let $G$ be a graph and $G_s$ be its spike graph. Then $b_T(G_s) = b(G)+1$.
\end{lemma}

\begin{proof}
    Recall that $b_T(G_s) = b(T(G_s))$. Let $H = T(G_s)$. Observe that the vertex set of $H$ can be partitioned as $V(H) = \mathcal{V} \uplus \mathcal{E} \uplus \mathcal{L} \uplus \mathcal{E'}$, where $\mathcal{V} = \{v_1, v_2, \dots, v_n\}$ and $\mathcal{E} = \{e_1, e_2, \dots, e_m\}$ correspond to the vertex set and the edge set of $G$, respectively, while $\mathcal{L} = \{l_1, l_2, \dots, l_n\}$ and $\mathcal{E'} = \{e'_1, e'_2, \dots, e'_n\}$ correspond to the sets of newly added vertices and edges in $G_s$, respectively. For each $i \in [n]$, the edge $e'_i$ corresponds to $v_i l_i$ in $G_s$. See Fig.~\ref{fig:SpikeExmpl} for an illustration of a spike graph and its total graph.
    
    Let $B$ be an optimal burning sequence for $G$. Burning $H$ as per $B$ leaves only vertices in $\mathcal{L} \cup \mathcal{E'} \cup \mathcal{E}$ unburned after $|B|$ steps. But any such vertex is adjacent to some vertex in $\mathcal{V}$, and hence one additional step suffices to burn all remaining vertices. Therefore, we have $b_T(G_s) \le b(G)+1$.
    
    Now, it remains to show that $b_T(G_s) > b(G)$. For contradiction, assume that $b_T(G_s) = b(G) = k$. For each $i \in [n]$, let $X_i = \{v_i, l_i, e'_i\}$. Since $v_i$, $l_i$, and $e'_i$ form a clique in $H$, any burning sequence for $H$ cannot contain all three vertices of $X_i$; otherwise, the vertex from $X_i$ that appears last in the sequence would already have been burned by an earlier source, contradicting the definition of a burning sequence.

    The following claim follows from the observation that, for any $i \in [n]$, the set of all vertices equidistant from $v_i$ and $e'_i$ in $H$ is precisely $(\mathcal{E} \cup \mathcal{E'} \cup \{l_i\}) \setminus \{e'_i\}$ and the set of all vertices equidistant from $l_i$ and $e'_i$ is precisely $(\mathcal{V} \cup \mathcal{L}) \setminus \{l_i\}$. In other words, $\{z\in V(H): d_H(z,v_i) = d_H(z,e'_i)\} = (\mathcal{E} \cup \mathcal{E'} \cup \{l_i\}) \setminus \{e'_i\}$ and $\{z\in V(H): d_H(z,l_i) = d_H(z,e'_i)\} = (\mathcal{V} \cup \mathcal{L}) \setminus \{l_i\}$.
    
    \renewcommand\qed{$\hfill\square$}
    \begin{claim}\label{clm:EquidistantSource}
        In any burning sequence $\tilde{B}_T$ for $H$ of size $k'$, and for any $i \in [n]$,
        \begin{itemize}
            \item if $v_i, e'_i \in S_{k'}(\tilde{B}_T) \cap B_c(s)$ for some $s \in \tilde{B}_T$, then $s \in (\mathcal{E} \cup \mathcal{E'} \cup \{l_i\}) \setminus \{e'_i\}$;
            \item if $l_i, e'_i \in S_{k'}(\tilde{B}_T) \cap B_c(s)$ for some $s \in \tilde{B}_T$, then $s \in (\mathcal{V} \cup \mathcal{L}) \setminus \{l_i\}$.
        \end{itemize}
    \end{claim}
    
    The following claim establishes a `specific' optimal burning sequence for $H$.
    
    \begin{claim}\label{clm:SpecialSequence}
        There exists an optimal burning sequence for $H$ that contains only the vertices from $\mathcal{V} \cup \mathcal{E}$ in the first $k-1$ positions.
    \end{claim}

    \begin{proof}
        Let $\tilde{B}_T = (b_1,b_2,\ldots,b_k)$ be an optimal burning sequence for $H$. Now, for every $j \in [k-1]$, if $b_j \in X_i \setminus \{v_i\}$ for some $i \in [n]$, then we replace $b_j$ in $\tilde{B}_T$ as follows. If $v_i \notin S^-_{j-1}(\tilde{B}_T)$ and $v_i$ is not in $\tilde{B}_T$, then $b_j = v_i$. If not, $b_j$ is replaced by a vertex in $(\mathcal{V} \cup \mathcal{E}) \cap S_{j}(\tilde{B}_T)$ (non-empty since $j < k = b(H)$). Since for any positive integers $r$ and $\ell$, $N_H^r[l_{\ell}] \subset N_H^r[v_{\ell}]$ and $N_H^r[e'_{\ell}] \subset N_H^r[v_{\ell}]$, the replacement of $l_i$ or $e'_i$ with $v_i$ is valid and ensures the same length. Thus, modified $\tilde{B}_T$ satisfies the claim.
    \end{proof}
    
    Let $B_T = (b_1, b_2, \dots, b_k)$ be an optimal burning sequence for $H$ such that for any $i \in [k-1]$, we have $b_i \in \mathcal{V} \cup \mathcal{E}$. The existence of $B_T$ follows from Claim~\ref{clm:SpecialSequence}.

    \begin{claim}\label{clm:AllFromVAndE}
        $S_k(B_T) \cap \mathcal{V} \neq \emptyset$.
    \end{claim}

    \begin{proof}
        Suppose, for contradiction, $S_k(B_T) \cap \mathcal{V} = \emptyset$. This implies $S_k(B_T) \subset \mathcal{L} \cup \mathcal{E'} \cup \mathcal{E}$. Now, let $A = B_T$. For each $i \in [k-1]$, if $b_i \in \mathcal{E}$, then replace $b_i$ in $A$ with a neighbor of $b_i$ in $\mathcal{V} \cap S_i(A)$ if it exists and any vertex of $\mathcal{V} \cap S_i(A)$ otherwise. Note that $\mathcal{V} \cap S_i(A) \neq \emptyset$ since $i<k$. Let the modified $A$ be $(a_1, a_2, \dots, a_k)$. If we burn $H$ as per $A$, then after $k-1$ steps, every unburned vertex lies in $\mathcal{L} \cup \mathcal{E'} \cup \mathcal{E}$. Therefore, since $d_G(a_i, v) = d_H(a_i, v)$ for every $v \in V(G)$, $(a_1, a_2, \dots, a_{k-1})$ is a burning sequence for $G$ of length $k-1$, a contradiction to $b(G)=k$.
    \end{proof}
    \renewcommand\qed{$\hfill\blacksquare$}

    Now, by Claim~\ref{clm:AllFromVAndE}, there exists a vertex $v_j \in \mathcal{V} \cap S_k(B_T)$ for some $j \in [n]$. Since the first $k-1$ vertices in $B_T$ belong to $\mathcal{V} \cup \mathcal{E}$, it follows that $e'_j \in S_k(B_T)$ as well. Thus both $v_j$ and $e'_j$ are in $S_k(B_T)$. Suppose that $b_k \neq l_j$. Then, since $v_j, e'_j \in S_k(B_T)$, the vertex $l_j$ remains unburned after $k$ steps, contradicting the fact that $B_T$ is a burning sequence for $H$. Hence $b_k = l_j$.

    Note that the same argument applies to any vertex $v_p$ in $\mathcal{V} \cap S_k(B_T)$. However, since $b_k \neq l_p$ for $p \neq j$, it follows that $v_j$ is the unique vertex in $\mathcal{V} \cap S_k(B_T)$. Moreover, since $v_j, e'_j \in S_k(B_T)$ and $e'_j \notin B_T$, we have $v_j, e'_j \in B_c(s)$ for some burning source $s \in \mathcal{V} \cup \mathcal{E}$. Hence, by Claim~\ref{clm:EquidistantSource}, we must have $s \in \mathcal{E}$.

    We now construct a vertex sequence $B'_T = (b'_1, b'_2, \dots, b'_{k-1})$ from $B_T$ as follows. For each $i \in [k-1]$, if $b_i \in \mathcal{V}$, set $b'_i = b_i$. Otherwise $b_i \in \mathcal{E}$, and we choose $b'_i$ to be a neighbor of $b_i$ in $\mathcal{V}$ that is closer to $v_j$ (breaking ties arbitrarily). We now burn $H$ as per $B'_T$. Since $s \in \mathcal{E}$, it follows that $v_j \in S_{k-1}(B'_T)$. Furthermore, as $v_j$ is the unique vertex in $\mathcal{V} \cap S_k(B_T)$, burning $H$ as per $B'_T$ leaves only vertices in $\mathcal{L} \cup \mathcal{E'} \cup \mathcal{E}$ unburned after $k-1$ steps. Hence $B'_T$ is a burning sequence for $G$ of length $k-1$, contradicting the assumption that $b(G) = k$. Therefore, $b_T(G_s) > b(G)$, completing the proof.
\end{proof}

Now, we consider the complexity of the total burning problem (i.e., \BNP{} on the total graphs). For any class of graphs $\mathcal{C}$, let $\mathcal{C}_s=\{G_s:G\in \mathcal{C}\}$, where $G_s$ is the spike graph of $G$. It follows from Lemma~\ref{lem:TotalSpike} that if the \BNP\ is \NPC\ for a graph class $\mathcal{C}$, then the total burning number problem is \NPC\ for any graph class $\mathcal{C}'$ such that $\mathcal{C}_s\subseteq \mathcal{C'}$. In particular, let $\mathcal{C}$ be the collection of bounded degree trees (i.e., trees having their maximum degree bounded by a constant). Since the spike graphs of bounded degree trees are again bounded degree trees, we have $\mathcal{C}_s\subseteq \mathcal{C}$. Therefore, as the \BNP\ is \NPC\ for bounded degree trees, we have the following theorem. 

\begin{theorem}
    \label{thm:totalhard}
    The total burning problem is \NPC\ for bounded degree trees. Equivalently, the \BNP\ is \NPC\ for the total graphs of bounded degree trees.
\end{theorem}

\section{Conclusion}

This work advances the understanding of graph burning from both algorithmic and structural perspectives. Our results strengthen the complexity landscape of the \BNP{}, establish an improved upper bound for the burning number of $P_k$-free graphs together with improved algorithms for constructing burning sequences, reveal fundamental relationships between the classical burning number and its edge and total variants, including the resolution of a conjecture on total burning, and determine the computational complexity of these variants on special graph classes. Several interesting questions remain open. It would be worthwhile to characterize the graphs for which the parameters $b(G)$, $b_L(G)$, and $b_T(G)$ coincide, either simultaneously or pairwise. More broadly, resolving the burning number conjecture remains one of the central open problems in the study of graph burning.

\bibliographystyle{SK}
\bibliography{main}

\end{document}